\pdfoutput=1
\documentclass[a4paper,11pt]{article}
\usepackage[english]{babel}
\usepackage{amsfonts}
\usepackage{amsthm}
\usepackage{amsmath}
\usepackage{amssymb}
\usepackage{enumerate}
\usepackage{mathtools}
\usepackage{tikz}
\usepackage{tikz-cd}
\usepackage[all]{xy}
\usepackage{graphicx}
\usepackage{amsmath}
\usepackage{adjustbox}
\usepackage{graphicx}
\usepackage{extpfeil}
\usepackage{comment}
\usepackage{mathabx}
\usepackage{float}
\usepackage[labelformat=empty]{caption}
\usepackage[toc]{appendix}
\usepackage[left=3.2cm,top=3cm,right=3.2cm,bottom=3cm]{geometry}
\usepackage{hyperref}
\usepackage{resizegather}
\usepackage[activate={true, nocompatibility}, final, tracking=true, kerning=true, spacing=true, factor=1100, stretch=10, shrink=10]{microtype}
\usepackage{authblk}
\usepackage{dsfont}
\usetikzlibrary{matrix}

\usepackage{ulem}

\usepackage{color} %CLASHES with hyperref
\definecolor{darkgreen}{rgb}{0.0, 0.7, 0.0}
\newenvironment{??}{\noindent \color{darkgreen}{\bf ???:} \footnotesize}{}
\definecolor{cyan}{cmyk}{1,0,0,0}

\newcommand{\bdg}{\begin{dg}}

\theoremstyle{plain}
\newtheorem{thm}{Theorem}[section]
\newtheorem{coroll}[thm]{Corollary}
\newtheorem{defn}[thm]{Definition}
\newtheorem{lemma}[thm]{Lemma}

\newtheorem{prop}[thm]{Proposition}

\newtheorem{remark}[thm]{Remark}

\newtheorem*{thm*}{Main Theorem}
\newtheorem{notn}[thm]{Notation}
\newtheorem{context}[thm]{Context}

\newcommand{\bseries}[1]{ [\hspace{-0,5mm}[ {#1} ]\hspace{-0,5mm}] }

\tikzset{
  symbol/.style={
    draw=none,
    every to/.append style={
      edge node={node [sloped, allow upside down, auto=false]{$#1$}}}
  }
}
\microtypecontext{spacing=nonfrench}

\DeclareMathOperator{\colim}{colim}
\newcommand{\cO}{\mathcal{O}}

\newcommand{\cE}{\mathcal{E}}
\newcommand{\cF}{\mathcal{F}}
\newcommand{\cM}{\mathcal{M}}
\newcommand{\on}{\operatorname}

\newcommand{\Spec}{\on{Spec}}

\usepackage{color} %CLASHES with hyperref
\definecolor{darkgreen}{rgb}{0.0, 0.7, 0.0}

\begin{document}
\title{\textbf{Geometry of the logarithmic Hodge moduli space}}
\author{Mark Andrea de Cataldo and Andres Fernandez Herrero,
\\
with an Appendix joint with Siqing Zhang}
\date{}
\maketitle
\begin{abstract}
We show the  smoothness over the affine line of the Hodge moduli space of logarithmic $t$-connections of coprime rank and degree on a smooth projective curve with geometrically integral fibers over an arbitrary Noetherian base.  When the base is a field, we also prove that the Hodge moduli space is geometrically integral. 
Along the way, we prove the same results for
the corresponding moduli spaces of logarithmic Higgs bundles and
of logarithmic  connections. We use smoothness to derive specialization isomorphisms on the \'etale cohomology rings of these moduli 
spaces; this includes the special case when the base is of mixed characteristic.  In the special case where the base is a separably closed field of positive characteristic, we show that these isomorphisms are filtered isomorphisms for the perverse filtrations associated with the corresponding Hitchin-type morphisms.

\end{abstract}

\begin{footnotesize}2020 Mathematics Subject Classification: 14D20 (primary), 14D23, 14F20 (secondary)
\end{footnotesize}

\tableofcontents

\begin{section}{Introduction}\label{sec:intro}
Let $C$ be a compact Riemann surface. The Non Abelian Hodge Theorem 
(NAHT) of
Simpson, Corlette and others (see \cite{simpson-repnII} and references therein), yields a canonical homeomorphism between three different moduli spaces of objects on the curve, namely, the moduli spaces of:
semistable Higgs bundles of fixed rank $n$ and degree $d=0$ (Higgs moduli space);
 algebraic flat connections of rank $n$ (de Rham moduli space); 
 representations of the fundamental group of the curve into the  general linear group 
${\rm GL}_n$ (Betti moduli space).

 In particular, the cohomology rings of these moduli spaces are canonically isomorphic, 
 a fact that we may call Cohomological NAHT. Following a suggestion by Deligne, 
 Simpson has introduced the moduli space of semistable  $t$-connections 
 (Hodge moduli space) on the curve, which
interpolates between the  moduli space of Higgs bundles 
(set $t=0$)
and the one of algebraic flat connections (set $t=1$).
The Hodge moduli space is topologically trivial over the affine line 
(corresponding to the parameter $t$) 
and one can view this triviality as an incarnation
of the NAHT.

The NAHT on a curve over the complex numbers has no direct analogue for curves over 
fields of positive characteristic. In this context, there is a Frobenius-twisted NAHT 
(cf. \cite{ogus-vologodsky}, \cite{groechenig-moduli-flat-connections}, \cite{chen-zhu}),
but this does not identify  moduli spaces. Absent such a NAHT, 
one can ask whether one still has an isomorphism
between the cohomology rings of the Higgs and de Rham moduli spaces.
Under natural conditions of coprimality involving rank, degree and characteristic, the smoothness of the Hodge moduli space and the isomorphism on cohomology rings have been addressed in \cite{decataldo-zhang-nahpostive}.

In this paper, we study logarithmic $t$-connections on a curve,
i.e. $t$-connections with at most simple poles
along a fixed effective reduced divisor on the curve. We focus on the
 case of coprime rank $n$ and degree $d.$ 
 
We provide an explicit treatment of the deformation theory of $t$-connections on curves 
leading to the smoothness of the Hodge/Higgs/de Rham moduli spaces.
As an application, we construct  in the case of coprime rank and degree  a canonical isomorphisms between the cohomology rings of these three moduli spaces.
 
There seems to be no complete study in the literature concerning the smoothness of the moduli of logarithmic $t$-connections, 
one that constructs  modules of obstructions, obstructions classes,
and provides an explicit criterion for smoothness. The rest of this paragraph is devoted to summarize the literature we are aware of on the subject.
A classical reference
for logarithmic Higgs bundles is \cite[\S6]{nitsure-higgs} where the dimension 
of tangent spaces is computed and smoothness in the coprime case 
can be derived indirectly, by use of the BNR correspondence (see \cite[\S2.1]{de-cataldo-support-sln}, for example).  The papers \cite{biswas-ramanan-infinitesimal, markman-spectral} have treatments of deformation theory of Hitchin pairs including descriptions of tangent spaces and obstructions, which are used in \cite[Cor. 7.9]{markman-spectral} to prove the smoothness of the moduli space of stable (logarithmic) Higgs bundles. \cite{yokogawa-infinitesimal-higgs}
deals with moduli of parabolic Higgs bundles and proves it is integral, normal and smooth at parabolic stable points; for special parabolic data, i.e. trivial filtration and weight adapted to the degree, this result specializes to our picture for 
logarithmic Higgs bundles. The paper \cite{boden_yokogawa} has a treatment of parabolic Higgs bundles when the rank is $2$, and shows that the moduli space is simply connected \cite[Thm. 4.2]{boden_yokogawa}. For a treatment of deformation theory of Hitchin pairs over the complex numbers using dgla language 
see  \cite[Thm. 5.3]{martinego-infinitesimal}. A classical reference for logarithmic connections over the complex numbers is \cite{nitsure-log-connections},
where the tangent space of the moduli problem is described as the hypercohomology of a complex, but where there is no treatment of obstructions. In the case of logarithmic connections with parabolic structure over the curve $\mathbb{P}^1$, the smoothness and irreducibility of the moduli space for generic stability parameters was shown in \cite[Thm. 1.1]{inaba_iwasaki_saito}. There is a treatment of deformation theory of parabolic logarithmic connections in \cite{inaba_riemann_hilbert}, where also the irreducibility of the moduli of parabolic logarithmic connections is proven \cite[Prop. 5.2, Prop. 5.3, Prop. 5.4]{inaba_riemann_hilbert}. \cite[\S5.3]{hao-lambda-modules-dm-stacks} discusses the deformation theory in the more general setup of $\Lambda$-modules on Deligne-Mumford stacks, describing a tangent-obstruction theory in this context,
but not addressing smoothness nor integrality questions. The article \cite{alfaya-oliveira-lie-algerbroids} discusses deformation theory and proves smoothness of the moduli space of modules for a Lie algebroid on a smooth projective curve over $\mathbb{C}$ in the coprime case. This includes the moduli of $t$-connections as a special case.

Let us introduce  the setup of this paper. Fix a family $C_B/B$
of smooth projective geometrically integral curves over a Noetherian base scheme $B,$ which could be, for example, a field, a DVR of equal or of mixed characteristic, etc.
Let $D_B/B$ be a relative strictly effective reduced divisor on the family of curves. We fix the rank $n$ and the degree $d$ of $t$-connections (the rank and degree of the underlying vector bundles) and we assume that they are coprime
$g.c.d.(n,d)=1.$

In this paper, we  study the Hodge moduli space
$M{\rm{Hodge}}_{C_B}^{ss} \to \mathbb{A}^1_B$
of  semistable logarithmic $t$-connections on $C_B/B$ of coprime rank $n$  and degree $d,$ 
with simple poles along $D_B/B.$  Our first main result is the following.

\begin{thm} \label{thm:smooth}
Assume that $n$ and $d$ are coprime and the fibers of the divisor $D$ are nonempty. The structural morphism (cf. Notation \ref{notn:tau_B})
$\tau_B:M{\rm{Hodge}}_{C_B}^{ss} \to \mathbb{A}^1_B$ 
is smooth. For each $a \in \mathbb{A}^1_B$, the fiber $(M{\rm{Hodge}}_{C_B}^{ss})_a$ is geometrically integral of dimension $n^2(2g-2+{\rm deg}(D_a))+1$.
 The same is true  for the fibers over the points $b \in B$ of the  morphisms  
$v_{{\rm{Higgs}},B}$ (\ref{notn:vhiggsB}), 
$v_{{\rm de \, Rham},B}$ (\ref{notn:vdrB}) and 
(here add $+1$ to the dimension)
$v_{{\rm{Hodge}},B}$ (\ref{notn:vHodgeB}).
\end{thm}
The proof of Theorem \ref{thm:smooth}) consists of first studying the deformation theory of $t$-connections
and then proving it is unobstructed. 
The vanishing of the obstruction class is proved using a degeneration argument from de Rham to Higgs inside of Hodge,
involving a vanishing theorem that makes essential use of the nonemptiness of the divisor of poles
$D_B/B$ on the geometric fibers of $C_B/B.$ In fact, in Proposition \ref{prop: smoothness of the stack} we show that under the same nonemptiness assumption, the Hodge stack of semistable objects is smooth over $\mathbb{A}^1_B$
without imposing any conditions on degree and rank. In the non coprime case, while the stack is smooth, the moduli space is usually singular; see  Remark \ref{stsmoospno}.

In \S\ref{smooth:nopoles}, we complement Theorem \ref{thm:smooth} by proving
a similar smoothness assertion in the case without poles, under necessary and thus natural numerical conditions.

We offer two applications (Theorem \ref{thm:spk} and Theorem \ref{thm:spv})
of the smoothness result (Theorem \ref{thm:smooth}) that relate to each other the cohomology rings of the Hodge, Higgs and de Rham moduli spaces.
The paper \cite{decataldo-zhang-nahpostive} 
proves  a weaker version of  Theorem \ref{thm:spk} 
in the case without poles, when $B$ is a field of positive characteristic 
and the rank and degree are subject to necessary, thus natural, conditions.

 Theorem \ref{thm:spk} and Theorem \ref{thm:spv} could be viewed as the cohomological 
shadow of a currently non-existing logarithmic NAHT in arbitrary, even mixed characteristic.
Even in the case of curves over the complex numbers, 
it is not clear to us how the moduli space of logarithmic $t$-connections would fit into the context of the parabolic NAHT of Simpson
and Mochizuki; see  \cite{simpson-harmonic-noncompact}, \cite[Cor. 1.5]{mochizuki-harmonic-tame-ii}.

The first application Theorem \ref{thm:spk} (see the companion diagram (\ref{eq:qapl})) is for the case when $B=Spec(k)$ is the spectrum of a separably closed field.
It shows that in the coprime case with poles, the natural restriction morphisms
on  cohomology rings (decorations omitted)  $H^*(M{\rm Higgs}) \leftarrow H^*(M{\rm Hodge}) \to H^*(M{\rm de Rham})$ are isomorphisms.  In fact, it shows that the specialization morphism
relating $H^*(M{\rm Higgs})$ and $H^*(M{\rm de Rham})$ satisfies the following:

\noindent 1) It is defined; a priori 
such a morphism does not exist due to lack of properness of the morphism $\tau_k: M{\rm Hodge} \to \mathbb{A}^1_k.$ We circumvent the lack of properness by means of a suitable completion
of the morphism $\tau_k.$

\noindent 2) It is an isomorphism.

\noindent 3)  If furthermore the field has positive characteristic, then all these isomorphisms are filtered isomorphisms for the perverse filtrations associated with
the various Hitchin-type morphisms in the picture (see \ref{sec:hitchmo}).
 
The second application Theorem \ref{thm:spv} (see the companion diagram (\ref{eq:iso055})) is for the coprime case with poles when $B = Spec(R)$ is the spectrum of a discrete valuation ring $R$. 
In this case, we have nine moduli spaces: Hodge/Higgs/de Rham
over the geometric closed point, over the geometric generic point and over the DVR. Their cohomology groups are related by restriction maps (denoted by the letters $\rho$ and $r$).
We prove the following:

\noindent a) All these restriction maps are isomorphisms.

\noindent b) The resulting collection of specialization maps are defined and are isomorphisms. 
 Again, we need to circumvent the lack of properness of various structural morphisms by means
of suitable compactifications. For technical reasons, if the DVR $R$ is of mixed characteristic $(0,p>0),$ we assume that $p>n,$ i.e that the rank  is smaller than $p.$

\noindent c) If furthermore, the DVR $R$ has equal positive characteristic, then all these isomorphisms are filtered isomorphisms for the perverse filtrations associated with
the various Hitchin-type morphisms.

 In both applications, we use compactification methods from 
\cite{decataldo-cambridge}, suitably generalized in \cite{decataldo-zhang-completion}. To this end, we need to recall in \S\ref{subs:1100} the construction
of suitably good compactifications of the relevant moduli spaces
given in \cite{decataldo-zhang-completion}.
Since such a compactification has not been constructed in the case of Hodge moduli spaces in
\cite{decataldo-zhang-completion}, we provide one here.

Finally, in the appendix \S\ref{appdx}  jointly written with S. Zhang, we provide a construction of the Hodge-Hitchin  morphism correcting minor inaccuracies in the literature.

{\bf Acknowledgments.} We thank Roberto Fringuelli, Jochen Heinloth, Pengfei Huang, Georgios Kydonakis, Mirko Mauri, Hao Sun, Siqing Zhang, and Lutian Zhao for useful conversations. We would also like to thank the referee for carefully reading the manuscript.
The first-named author has been partially supported by NSF grants
1901975 and DMS-2200492, and by a Simons Fellowship in Mathematics Award n. 672936.  
\end{section}

\begin{section}{Moduli stacks/spaces with poles}\label{section:moduli}

\subsection{Notation and setup}\label{subsection:setup}
We work over a Noetherian scheme $B$. Let $\pi: C \to B$ be a smooth proper morphism of schemes with geometrically integral fibers of dimension 1. We refer to $C\to B$
simply  as a curve
over $B$, and  we denote it by $C_B.$ Let $D \hookrightarrow C$ be a relative Cartier divisor such that every geometric $B$-fiber of $D$ is nonempty and reduced.

We fix once and for all two integers $n>0$ -the rank- and $d \in \mathbb{Z}$
-the degree-. 
We assume  that $n$ and $d$ are  coprime.

For every morphism of schemes  $S\to B,$  we denote the corresponding fiber product  $C\times_B S \to S,$ simply by $\pi_S: C_S \to S$. For example if $\overline{b}\to b \to  B$ is a geometric point over a Zariski point $b$ of $B$,
then $C_b$ denotes the fiber over $b$ and $C_{\overline{b}}$ denotes the corresponding  geometric fiber.

We will often work over the base $\mathbb{A}^1_B = \underline{\rm{Spec}}_{B}(\mathcal{O}_{B}[t])$, equipped with the action of the multiplicative group scheme  $\mathbb{G}_{m,B}$ that assigns $t$ weight $1$. For any $\mathbb{A}^1_{B}$-scheme $S$, we shall denote by $t_{S}$ the global section of the structure sheaf $\mathcal{O}_{S}$ obtained by pulling back $t$.

If $Spec(A) \to \mathbb{A}^1_{B}$ is a morphism from an affine scheme, we might abuse notation and write $\pi_{A}: C_{A} \to Spec(A)$ and $t_{A} \in A$ as replacement of the notation described above.

If $S\to B$ is a morphism of schemes and $\mathcal G$ is an object over $B$ (a scheme over $B$,
an $\mathcal{O}_B$-module,  an $\mathcal{O}_X$-module with $X$ a scheme over $B$, etc.), then $\mathcal{G}_S$ denotes the pulled-back object via the morphism $S \to B$.

\begin{subsection}{The Hodge moduli stack/space}\label{subsection:Hodge with poles}

\begin{defn}\label{def:Hodge}
We denote by $\mathcal{M} {\rm Hodge}_{C_B}^{ss}   \to \mathbb{A}^1_B$ the moduli stack of (slope) semistable rank $n$ and degree $d$ logarithmic $t$-connections with poles along $D$. As a pseudofunctor, it sends an $\mathbb{A}^1_{S}$-scheme $S \to \mathbb{A}^1_{S}$ to the groupoid of pairs $(\mathcal{F}, \nabla)$, where:
\begin{enumerate}[(a)]
    \item $\mathcal{F}$ is a vector bundle of rank $n$ on $C_S$ such that its restriction to each geometric fiber of the morphism $C_{S} \to S$ has degree $d$.
    \item $\nabla: \mathcal{F} \to \mathcal{F} \otimes_{\mathcal{O}_{C_S}} \omega_{C_S/S}(D_S)$ is a logarithmic $t_{S}$-connection with (at most simple)
    poles allowed at the pull-back $D_S$ of $D$.
    \item The restriction of the pair $(\mathcal{F}, \nabla)$ to each geometric fiber $C_{s}$ of the morphism $C_{S} \to S$ is a semistable $t_{s}$-connection.
\end{enumerate}
\end{defn} 

Since $n$ and $d$ are  coprime, every semistable pair $(\mathcal{F}, \nabla)$ is in fact stable. 

When dealing with $t$-connections, 
one uses the sheaf of rings on $C_B \times_{B}\mathbb{A}^1_{B}$ given by the Rees degeneration with respect to the order filtration of the enveloping algebra of differential operators associated with the Lie algebroid of relative vector fields on $C \to B$ (see \cite[$\tau$-connections, pg. 87]{Simpson-repnI} in characteristic $0$ and \cite[\S2.2]{langer2014semistable} for the enveloping algebra over a general base).

The degeneration takes the following shape: at the section $0_B,$ it is the push-forward of the  algebra of functions on the total space of the relative cotangent bundle $\omega_{C_B/B}$, whose modules are Higgs bundles on $C_B;$ at the section $1_B,$ it is the sheaf of crystalline differential operators, whose modules are flat connections
on $C_B/B$.

In this paper, we deal with poles,
so that instead we use the Rees degeneration of the universal enveloping algebra of differential operators associated with the Lie algebroid of relative vector fields 
$(\omega_{C_B/B}(D))^\vee$ on $C_B \to B$ vanishing at the poles $D$. 

\begin{notn} \label{notn:tau_B}
We denote by $\tau_B: M{\rm{Hodge}}_{C_B}^{ss} \to \mathbb{A}^1_B$ the corresponding quasiprojective moduli space of rank $r$ and degree $d$ semistable logarithmic $t$-connections constructed
over a base $B$ a $\mathbb{C}$-scheme of finite type in \cite{Simpson-repnI}  by using Geometric Invariant Theory, and more recently over a general Noetherian base $B$ in \cite{langer-moduli-lie-algebroids}.
\end{notn}

The natural  $\mathbb{A}^1_B$-morphism $\mathcal{M}{\rm{Hodge}}_{C_B}^{ss} \to M{\rm{Hodge}}_{C_B}^{ss}$
exhibits $M{\rm{Hodge}}_{C_B}^{ss}$ as an adequate moduli space (as in \cite{alper_adequate}). In \S\ref{subs:redsmst},  by using the coprimality of rank and degree, we  show that this is a good moduli space (as in \cite{alper-good-moduli}).

\begin{notn} \label{notn:vHodgeB}
We denote the natural morphism obtained by composing with the projection onto $B$ by $v_{Hodge,B}: M{\rm{Hodge}}_{C_B}^{ss} \to B$.
\end{notn}

Both the stack $\mathcal{M}{\rm{Hodge}}_{C_B}^{ss}$ and the scheme $M{\rm{Hodge}}_{C_B}^{ss}$ are of finite type over $\mathbb{A}^1_B$. The fact that the stack $M{\rm{Hodge}}_{C_B}^{ss}$ is locally of finite type follows from the GIT setup (or by e.g. \cite[Prop 2.2.2]{torsion-freepaper}). On the other hand $M{\rm{Hodge}}_{C_B}^{ss}$ is quasiprojective over $\mathbb{A}^1_B$, and so it follows that the stack $\mathcal{M}{\rm{Hodge}}_{C_B}^{ss}$ with moduli space $M{\rm{Hodge}}_{C_B}^{ss}$ is also quasicompact. 

The group scheme $\mathbb{G}_{m,B}$ acts on $M{\rm{Hodge}}_{C_B}^{ss}$ by scaling the universal logarithmic $t$-connection; this induces an action of $\mathbb{G}_{m,B}$ on the moduli space $M{\rm{Hodge}}_{C_B}^{ss}$. Both morphisms $M{\rm{Hodge}}_{C_B}^{ss} \to \mathbb{A}^1_B$ and $\mathcal{M} {\rm Hodge}_{C_B}^{ss}   \to \mathbb{A}^1_B$ are $\mathbb{G}_{m,B}$-equivariant.

\end{subsection}
 \begin{subsection}{The Higgs and de Rham moduli stacks/spaces}\label{subsection:Higgs with poles}
  \begin{defn}
  The Higgs moduli stack $\mathcal{M}{\rm{Higgs}}_{C_B}^{ss}$ is defined by the following Cartesian diagram
\begin{figure}[H]
\centering
\begin{tikzcd}
  \mathcal{M}{\rm{Higgs}}_{C_B}^{ss} \ar[r, symbol= \hookrightarrow] \ar[d] &  \mathcal{M}{\rm{Hodge}}_{C_B}^{ss}
  \ar[d] 
  \\ 
  B (=0_B) \ar[r, symbol= \xhookrightarrow{t_B=0}] &  \mathbb{A}^1_B.
\end{tikzcd}
\end{figure}
For every $B$-scheme $S\to B$, it classifies pairs $(\mathcal{F}, \nabla)$ with
$\mathcal{F}$ a vector bundle of rank $r$ and degree $d$ on $C_S$ and $\nabla$ a logarithmic Higgs field
$\nabla:\mathcal{F} \to \mathcal{F} \otimes_\mathcal{O_{C_S}} \omega_{C_S/S}(D_S)$ with poles at $D_S.$
 \end{defn}
  \begin{notn} \label{notn:vhiggsB}
   We denote by $v_{Higgs,B} : M{\rm{Higgs}}_{C_B}^{ss} \to B$ the quasiprojective moduli space of semistable logarithmic Higgs bundles with poles along $D$ constructed using GIT as recalled in Notation \ref{notn:tau_B}.
  \end{notn}
  Since the formation of good moduli spaces is compatible with arbitrary base change, by Lemma \ref{lemma:gm gerbe} the canonical morphism for moduli spaces below is an isomorphism (see also \cite[Thm. 1.1]{langer-moduli-lie-algebroids}):
 \begin{equation}\label{eq:higermod}
 \xymatrix{
  M{\rm{Higgs}}_{C_B}^{ss} 
  \ar[r]^-\simeq
  &
 ( M{\rm{Hodge}}_{C_B}^{ss})_{0_B}.
 }
 \end{equation}
 
 \begin{defn}
  The de Rham stack $\mathcal{M}{\rm de \, Rham}_{C_B}^{ss}$ is defined by the following Cartesian diagram
\begin{figure}[H]
\centering
\begin{tikzcd}
  \mathcal{M}{\rm de \, Rham}_{C_B}^{ss} \ar[r, symbol= \hookrightarrow] \ar[d] &  \mathcal{M}{\rm{Hodge}}_{C_B}^{ss}
  \ar[d] 
  \\ 
  B (=1_B) \ar[r, symbol= \xhookrightarrow{t_B=1}] &  \mathbb{A}^1_B.
\end{tikzcd}
\end{figure}
It classifies pairs $(\mathcal{F}, \nabla)$ with $\nabla$ a logarithmic connection with poles at $D$.
 \end{defn}
 
  \begin{notn} \label{notn:vdrB}
We denote by $v_{{\rm de \, Rham},B} : M{\rm de \, Rham}_{C_B}^{ss} \to B$ the corresponding quasiprojective moduli space of semistable logarithmic connections with poles along $D$ constructed using GIT.
  \end{notn}
  When  restricted over the open $\mathbb{G}_{m,B} \subseteq \mathbb{A}^1_B,$ both  the Hodge moduli stack and space are fiber products over by   $\mathbb{G}_{m,B}$ over  $B$
of the de Rham moduli:  a $t$-connection with $t$ invertible $(\mathcal F, \nabla)$
can be rescaled to a connection $(\mathcal F, \frac{1}{t} \nabla).$ This trivialization is $\mathbb{G}_{m,B}$-equivariant:
\begin{equation}\label{eq:trivl}
\xymatrix{
(M{\rm{Hodge}}_{C_B}^{ss})_{\mathbb{G}_{m,B}} 
\ar[r]^-\simeq
&
M{\rm de \, Rham}_{C_B}^{ss} \times_B \mathbb{G}_{m,B},
&
(\mathcal F, \nabla) \mapsto ((\mathcal F, \frac{1}{t} \nabla),t).
}
\end{equation}

In view of this triviality over $\mathbb{G}_{m,B},$
one can show directly that there is an isomorphism 
  \begin{equation}\label{eq:fiberderh}
 \xymatrix{
  M{\rm de \, Rham}_{C_B}^{ss} 
  \ar[r]^-\simeq
  &
 ( M{\rm{Hodge}}_{C_B}^{ss})_{1_B}.
 }
 \end{equation}
 This isomorphism (\ref{eq:fiberderh}) holds without having to assume the coprimality 
of rank and degree.
At present, we ignore if, absent the coprimality assumption, the same is true for the Higgs moduli space. 

\end{subsection}

\begin{subsection}{Hitchin-type morphisms and perverse filtrations}\label{sec:hitchmo}

In this section, we use \cite[esp. \S4,5]{langer-moduli-lie-algebroids}  as a reference, but we employ a notation closer to the one in \cite[esp. \S2.2]{decataldo-zhang-completion}.

\begin{notn}
The Higgs moduli space comes equipped with the Hitchin $B$-morphism:
\begin{equation}\label{eq:hitchinmo}
\xymatrix{
h_{{\rm{Higgs}},B} : M{\rm{Higgs}}_{C_B}^{ss} 
\ar[r]
&
A(C_B),
} 
\end{equation}
with target the vector bundle on $B$ with fibers 
$A(C_b):= \oplus_{j=1}^n H^0(C_b, (\omega_{C_b/b}(D_b))^{\otimes j}),$
and which assigns to a Higgs bundle its characteristic polynomial.
\end{notn}

\begin{notn}
If $B$ has characteristic $p>0,$ then there is the Hodge-Hitchin $B$-morphism: (see 
\S\ref{appdx}) 
\begin{equation}\label{eq:hodgemo1}
\xymatrix{
h_{{\rm{Hodge}},B} :   M {\rm{Hodge}}_{C_B}^{ss} 
\ar[r]
&
A(C_B^{(B)})   \times_{B} \mathbb{A}^1_B,
} 
\end{equation}
where $C_B^{(B)}= B\times_{B,fr_B} C_B$  (absolute Frobenius $fr_B: B \to B$) is the Frobenius twist of $C_B$ relative to $B,$
and the morphism assigns to a $t$-connection with poles the $p$-th root of the characteristic polynomial of its $p$-curvature. 
\end{notn}

Both the Hitchin morphism $h_{{\rm{Higgs}},B}$ \eqref{eq:hitchinmo} and the Hodge-Hitchin morphism $h_{{\rm{Hodge}},B}$ \eqref{eq:hodgemo1} are proper by \cite[Thm. 5.9]{langer-moduli-lie-algebroids}.

\begin{notn}
By restricting the Hodge-Hitchin morphism to connections with poles (i.e. $t=1$) we obtain the 
de Rham-Hitchin $B$-morphism (a.k.a. the $p$-Hitchin morphism; it is clearly proper):
\begin{equation}\label{eq:hodgemo2}
\xymatrix{
h_{{\rm{de\,Rham}},B} :   M {\rm{de\,Rham}}_{C_B}^{ss} 
\ar[r]
&
A(C_B^{(B)}).
} 
\end{equation}
\end{notn}

 By restricting the Hodge-Hitchin morphism to  logarithmic Higgs bundles, we obtain
 the Hitchin morphism to $A(C_B)$ post-composed with the relative Frobenius $B$-morphism $Fr_{A(C_B)};$
 see \cite[Lemma 4.3]{decataldo-zhang-completion}, which, while stated and proved in the case without poles, can be proved in the same way in the logarithmic case.

Each of these Hitchin-type morphisms induces an increasing 
filtration, called the perverse (Leray) filtration (cf. \cite[\S2.1]{decataldo-cambridge}) on the respective  $\overline{\mathbb{Q}}_\ell$-adic cohomology rings (with some decorations omitted)
\begin{equation}\label{eq:pf}
(H^*(M {\rm{Higgs}}), P_{\rm Higgs}), \,
(H^*(M {\rm{Hodge}}), P_{\rm Hodge}), \,
(H^*(M {\rm de \, Rham}), P_{\rm de\,Rham}).
\end{equation}

Since the relative Frobenius morphism $Fr_{A(C_B)}$ is a universal homeomorphism,
 the perverse filtrations on the cohomology ring $H^*(M {\rm{Higgs}})$ associated with the Hitchin morphism and with the Hodge-Hitchin morphism
 restricted to the Higgs moduli space coincide.
\end{subsection}
\end{section}

\begin{section}{Specialization morphisms}

The goal of this section is to remind the reader of the notions of specialization morphism and of its filtered counterparts given in \cite{decataldo-cambridge}, so that the two applications of Theorem \ref{thm:smooth} we give in this paper, namely, Theorems \ref{thm:spk}  and \ref{thm:spv}, can be stated.

The typical setup for  specialization morphisms is the one of a morphism to a DVR.  In Theorem \ref{thm:spk},
the DVR in question is the Henselization of the local ring at the origin of the affine line over an algebraically closed field $k,$ and the morphism to it is
the restriction of the structural morphism $\tau_k$
of the Hodge moduli space
in Notation \ref{notn:tau_B}. The specialization morphism
then relates the cohomology rings of the Higgs and de Rham
moduli space, with the one of the Hodge moduli space acting as an intermediary.

In Theorem \ref{thm:spv},
the DVR is arbitrary,
and the morphism is the morphism $\tau_B$
in Notation \ref{notn:tau_B}.
The specialization morphism
then relates the cohomology rings of the Hodge moduli spaces over the geometric closed and generic points.

\begin{subsection}{(Filtered) Specialization morphisms}\label{subsec:specmor}
A reference for this section is \cite{decataldo-cambridge}. We will freely employ the associated formalism of nearby/vanishing cycles \cite[XIII]{SGA7-2}. Specialization morphisms appear in the statements of Theorems
\ref{thm:spk} and \ref{thm:spv}.

  Let $(A, a, \overline{\alpha})$ be the spectrum of strictly Henselian DVR, together with its closed and geometric point 
  $i_a=i:a\to A,$ open point $\alpha,$ and a choice
  of geometric generic point $j_{\overline{\alpha}}= \overline{j}: \overline{\alpha} \to A$ induced by a separable closure of $k(\alpha).$ Fix a prime $\ell$ that is invertible in the residue field of $a$.

Let  $Y$ be a scheme and let $v_Y=v:Y\to A$ be a separated morphism of finite type. 
Let $D^b_c(Y)$  be the $\overline{\mathbb{Q}}_\ell$-constructible derived category
on $Y$. 
 We have the distinguished triangle 
   of functors $(i^*[-1],\psi_v [-1],\phi_v)): D^b_c (Y) \to D^b_c(Y_a)$, where $i:Y_{a} \to X$ is the closed embedding of the special fiber,
   and $\psi_v$ and $\phi_v $ are the nearby and vanishing cycle functors, with values supported both on $Y_a$
   and on $a$ depending on the context. In our notation, 
   $\psi_v [-1]$ and $\phi_v$ are $t$-exact.
   We have the base change morphism $i^*v_* \to v_*i^*.$

   Let $G \in D^b_c(Y)$. 
   We have the natural morphisms in  cohomology  (cf. \cite[II-6 p. 23]{SGA4.5}): (we omit pull-back notation on $G$)
   \begin{equation}\label{eq:twors}
   \xymatrix{
  H^*(Y_{a},G) 
  &
  H^*(Y=Y_A,G) 
  \ar[l]_-{r_a}
  \ar[r]^-{r_{\overline{\alpha}}}
  &
  H^*(Y_{\overline{\alpha}},G),
  }
  \end{equation}
  where we employ the letter $r$ to denote the pull-back/restriction in cohomology via the evident morphisms.
  If $G={\overline{\mathbb{Q}}}_{\ell,Y},$ then these are  morphisms of cohomology rings.
  
  The reference \cite[\S1-3]{decataldo-cambridge} works over the complex numbers with the classical topology. 
  As it is pointed out in the introduction to \cite{decataldo-cambridge} and  in \cite[\S 5]{decataldo-zhang-completion}, 
 \cite[\S1-3]{decataldo-cambridge} remain valid, with only calligraphic changes, in our set-up over a DVR. 
  
\begin{defn}[{\cite[Defn. 3.1.3]{decataldo-cambridge}}]
We say that the specialization morphism $sp_v(G)$ is defined if the base change morphism $bc^{i^*v_*}$ in \cite[Def. 3.1.3, based on diagram (42)]{decataldo-cambridge} is an isomorphism.
  In this case, the pull-back morphism $r_{a}$ 
  (\ref{eq:twors}) is invertible,
  in which case, we define the specialization morphism:
  \begin{equation}\label{eq:defsp}
  \xymatrix{
  sp_v:= sp_v(G)=   r_{\overline{\alpha}}   \circ  {{r_{\overline{a}}}}^{-1} :
   H^*(Y_{a},G)  
   \ar[r]
   &
    H^*(Y_{\overline{\alpha}},G).
   }
  \end{equation}
\end{defn}  
  
  If $v$ is proper, then the specialization morphism is defined by proper base change.
  If $v$ is not proper, then the specialization morphism can fail to be defined: e.g.
  when $v: Y := A \setminus \{a\} \to A.$
  
  For what follows, we refer to \cite[\S5.2 (Rectified perverse $t$-structure over a DVR)]{decataldo-zhang-completion}.
  In particular, we have O. Gabber's rectified perverse  $t$-structure on the $\overline{\mathbb{Q}}_\ell$-constructible derived category on $Y$.  One way to think of it is to view it, in first approximation,
 as gluing perverse sheaves on $Y_a$ to perverse sheaves on $Y_{\overline{\eta}}$ shifted by $[1].$
  We thus have the notion of the perverse filtration $P$ on the cohomology  $H^*(Y,G)$
  of a $\overline{\mathbb{Q}}_\ell$-constructible complex $G$  on $Y$.
  
  Let $f:X \to Y$ be a separated morphism of finite type. Let $v_X:= v_Y\circ f:   X \to A.$
  There is the notion of perverse Leray filtration $P^f$ relative to the morphism $f$ on the cohomology of 
  a $\overline{\mathbb{Q}}_\ell$-constructible complex $F$ on $X,$ which is defined to be the perverse filtration on the cohomology of the derived direct image $f_*F$  on $Y$, i.e.
  $(H^*(X,F), P^f):= (H^*(Y,f_*F),P).$ 
  
 When dealing with the cohomology ring of $X$,  for convenience, we number the perverse Leray filtration so that $1$ lands in the $0$-th graded subquotient.

  Let $F$ be a $\overline{\mathbb{Q}}_\ell$-constructible complex  on $X$.
  \cite[Def. 3.3.3, based on diagram (55)]{decataldo-cambridge} defines  the notion
  of filtered specialization morphism  for $F$ on $X$ and for the composition $X\to Y \to A:$
  \begin{equation}\label{eq:sppf}
  sp_v:  (H^*(X_a, F), P^{f_a}) \to
  (H^*(X_{\overline{\alpha}},F), P^{f_{\overline{\alpha}}})
  \end{equation}
    relative to the perverse Leray filtrations $P^{f_a}$ and $P^{f_{\overline{\alpha}}}.$
    Our notation here differs slightly from \cite{decataldo-cambridge},
    and we emphasize that we are considering  the specialization morphism for the morphism
    $v_Y : Y \to A$ for the derived direct image complex $f_*F$ on $Y,$ filtered by the perverse $t$-structure on $Y.$
    The special case $f={\rm Id}_Y$ gives the notion of filtered specialization morphism
    for $G$ on $Y$
    for the morphism $Y\to A;$  see \cite[(49)]{decataldo-cambridge}.
\begin{defn}[cf. {\cite[Defn. 3.2.3]{decataldo-cambridge}}]\label{filtsp}
We say that the filtered specialization morphism is  defined if  the two sequences, labelled by the integers, of  base change arrows on the left-hand-side column of \cite[(55)]{decataldo-cambridge} are invertible. In this case, we obtain the filtered morphism (\ref{eq:sppf}).
\end{defn}   
  
  \begin{remark} \label{remark: shfit by one}
  If the filtered specialization morphism is defined, then so is the specialization morphism,
  which is then the morphism underlying the filtered version.
  In the special case when the morphisms of type $\delta$ are isomorphisms, so that $r_a$ 
  is a filtered isomorphism, then we have $P^{f_a}(1) \xrightarrow{\sim} P^f \to P^{f_{\overline{\eta}}} (1)$.
%In the special case that the morphism $r_a$, besides being invertible, is filtered, because the morphisms
%   of type $\delta$ are isomorphisms, then we have $P^{f_a}(1) = P^f \to P^{f_{\overline{\eta}}} (1).$
  \end{remark}
  
  \end{subsection}
  
  \begin{subsection}{Specialization for  the Hodge moduli space over a field}\label{subsection:sphodge}
The purpose of this section is to introduce and discuss the commutative diagram
  (\ref{eq:qapl}), which we need to state (and to prove) Theorems \ref{thm:spk}. 
  
\begin{context}
 Let $B=Spec(k)$ be a separably closed field and let $C_k$ be our curve.
Consider the Hodge moduli space $\tau_k: M{\rm{Hodge}}_{C_k}^{ss} \to \mathbb{A}^1_k,$ together with its fibers $M{\rm{Higgs}}_{C_k}^{ss}$ over $0_k$, and $M\rm{de\,Rham}_{C_k}^{ss}$ over $1_k$.
\end{context} 

\begin{notn}
We denote by the same symbol $\tau_k$  the morphism obtained by base changing $\tau_k$
  via the morphism
  $Spec (\widetilde{\mathcal{O}_{0_k,\mathbb{A}^1_k}}) \to \mathbb{A}^1_k,$
  where  $\widetilde{\mathcal{O}_{0_k,\mathbb{A}^1_k}}$ is the strict Henselization
  of the local ring at the origin $0_k \in \mathbb{A}^1_k.$
\end{notn}

Let $\overline{\infty} \to \infty \in \mathbb{A}^1_k,$ be a fixed geometric generic point of the affine line $\mathbb{A}^1_k$
  induced by a choice of a separable closure of $k(\infty).$
  We  have
  the  nearby/vanishing-cycle functors $\psi_{\tau_k}$ and $\phi_{\tau_k}$ for this new morphism $\tau_k.$

Because of  the product structure  of the Hodge moduli space over $\mathbb{G}_{m,k} \subset \mathbb{A}^1_k,$ we have canonical isomorphisms:
\begin{equation}\label{eq:can identif}
  H^* (0_k, \psi_{\tau_k}{\tau_k}_* \overline{\mathbb{Q}}_\ell) =
   H^*((M{\rm{Hodge}}_{C_k}^{ss})_{\overline{\infty}}) \stackrel{\simeq}\leftarrow
   H^*(M{\rm de \, Rham}_{C_k}^{ss}),
  \end{equation}
where: the first equality is the classical and general fact that the cohomology 
of the nearby cycle functor applied to the derived direct image via $\tau_k$
  agrees with the cohomology of the geometric generic fiber; 
  the second identification is due to the aforementioned product 
  structure, in view of the natural morphism $\overline{\infty} \to k \to 1_k.$

  We have a commutative diagram, where the arrows are the morphisms induced by restriction/pull-back:
\begin{equation}\label{eq:qapl}
  \xymatrix{
  H^*(M{\rm{Higgs}}_{C_k}^{ss}) 
  \ar[d]^-=
  &
  H^*(M{\rm{Hodge}}_{C_k}^{ss})
  \ar[d] 
  \ar[r]^-{\rho_{1_k}}
  \ar[l]_-{\rho_{0_k}}
 &
 H^*(M{\rm de \, Rham}_{C_k}^{ss})
 \ar[d]^-\simeq_{(\ref{eq:can identif})}
\\
H^*( (M{\rm{Hodge}}_{C_k}^{ss})_{0_k}   )
 &
 H^*((M{\rm{Hodge}}_{C_k}^{ss})_{\widetilde{\mathcal{O}_{0_k, \mathbb{A}^1_k}}}) 
  \ar[l]_-{r_{0_k}} 
  \ar[r]^-{r_{\overline{\infty}}}
 &
 H^*((M{\rm{Hodge}}_{C_k}^{ss})_{\overline{\infty}}).
   }
\end{equation}

The  specialization morphism $sp_{\tau_k}$  associated with $\tau_k$ is defined 
if the  morphism  $r_{0_k}$ is an isomorphism so that we can set, by using the identification (\ref{eq:can identif}),
special to our situation:
\begin{equation}\label{eq:loij}
sp_{\tau_k}:=   r_{\overline{\infty}}  \circ r_{0_k}^{-1} :  H^*(M{\rm{Higgs}}_{C_k}^{ss})  \to H^*(M{\rm de \, Rham}_{C_k}^{ss}).
\end{equation}

The morphism $\tau_k$ is not proper, so that it is a priori unclear that the specialization morphism is defined. \cite[Thm. 3.5]{decataldo-zhang-nahpostive} shows that the restriction morphisms  $\rho_{0_k}$  and $\rho_{1_k}$
in (\ref{eq:qapl})
are isomorphisms in the case of $\text{char} (k)  >0,$  without poles, also under some suitable coprimality conditions.
The same proof works in the case with poles; see the proof of  Theorem \ref{thm:spk}.
 On the other hand, \cite{decataldo-zhang-nahpostive} does not address explicitly the existence and  properties of
the (filtered) specialization morphism; 
 Theorem  \ref{thm:spk} puts a remedy to  these omissions.
 \begin{thm}\label{thm:spk}
 Let $B=k$ be a separably closed field and let $C_k$ be our curve.  Assume that the rank $n$ and the degree $d$ are coprime, and the fibers of the divisor $D$ are nonempty.
 Then all the  morphisms in (\ref{eq:qapl}) are isomorphisms of cohomology rings,
 the specialization morphism $sp_{\tau_k}$ (\ref{eq:loij}) is defined, it is an isomorphism
 of cohomology rings
 and we have an identification 
 \begin{equation}\label{eq:loik}
sp_{\tau_k} = \rho_{1_k} \circ {\rho_{0_k}}^{-1}.
 \end{equation}
If, in addition,  ${\rm char} (k)>0,$ then all the morphisms in (\ref{eq:qapl}) and  (\ref{eq:loij})
are filtered isomorphisms for the respective perverse filtrations as in \S\ref{subsec:specmor}.
 \end{thm}

% However, by looking forward ahead for a moment, it may be useful to point out now that we shall 
% prove, as part of the forthcoming Theorem \ref{thm:spk}, that $\rho_{0_k}$
% and $r_{0_k}$  are isomorphisms, so that, the specialization morphism is defined and, moreover,
% we may also view, in our context, the specialization morphism as the more intuitive morphism:

% If the field $k$ has positive characteristic, then we have the perverse filtrations for the various Hitchin-type morphisms in the picture, and the forthcoming  Theorem \ref{thm:spk} is concerned with the filtered specialization morphism.

\end{subsection}

\begin{subsection}{Specialization for  the Hodge moduli space over a  DVR}\label{subsection:sphodged}
  The purpose of this section is to introduce and discuss the commutative diagram
  (\ref{eq:iso055}), which we need to state (and to prove) Theorem \ref{thm:spv}.

\begin{context}Let 
$(B,s,\overline{\eta})$  be the spectrum of a strictly Henselian   DVR with closed geometric  point 
  $s \in B$ and a choice of a geometric generic point $\overline{\eta} \to \eta \in B$.
  \end{context}

   The morphisms (\ref{notn:vHodgeB}), (\ref{notn:vhiggsB}) and (\ref{notn:vdrB})
   of type $v_{?,B}: M\text{?}\to B$ gives rise to possible specialization morphisms
   that we denote by $sp_{v_{{\text ?},B}}.$  As usual, each one is defined if and when the associated  pull-back/restriction morphism,
   denoted $r_{s},$ is an isomorphism, so that we can set
   $sp_{v_{?, B}}:= r_{\overline{\eta}} \circ r_{s}^{-1}: H^*(M\text{?}_{s}) 
   \to H^* (M\text{?}_{\overline{\eta}})$.  There are three potential versions of such specialization morphisms 
   of type $sp_{v_{\text{?},B}}$: the Hodge, the Higgs and the de Rham version:
   \begin{equation}\label{eq:3vb}
   sp_{v_{{\rm{Hodge}},B}}, \; sp_{v_{{\rm{Higgs}},B}}, \;sp_{v_{{\rm de \, Rham},B}}.
   \end{equation}
   
  Moreover, in the Hodge case, according to \S\ref{subsection:sphodge}, esp. (\ref{eq:qapl}), 
  we  have the possible  specialization morphisms $sp_{\tau_{s}}$
   and $sp_{\tau_{\overline{\eta}}}$ associated with the structural morphisms
   $\tau_{s}: M{\rm{Hodge}}^{ss}_{C_{s}}  \to \mathbb{A}^1_{s}$ and 
   $\tau_{\overline{\eta}}: M{\rm{Hodge}}^{ss}_{C_{\overline{\eta}}} \to \mathbb{A}^1_{\overline{\eta}}$.

We summarize the discussion above  via the natural commutative diagram of restri\-ctions/pull-backs and specializations,
 all of which are morphisms of cohomology rings:
 \begin{equation}\label{eq:iso055}
\xymatrix{
H^*(M{\rm{Higgs}}^{ss}_{ C_{s}}) 
\ar@/^2pc/[rr]^-{sp_{\tau_{s}}}_?
\ar@/_6pc/[dd]_-{sp_{\tau_{v_{{\rm{Higgs}},B}}}}^-?
&  
H^*(M{\rm{Hodge}}^{ss}_{ C_{s}})   
\ar[l]_-{\rho_{0_{s}}}
\ar[r]^-{\rho_{1_{s}}}
& 
H^*(M{\rm de \, Rham}^{ss}_{ C_{s}}) 
\ar@/^6pc/[dd]^-{sp_{\tau_{v_{{\rm de \, Rham},B}}}}_-?
\\
H^*(M{\rm{Higgs}}^{ss}_{ C_B})
\ar[d]^-{r_{\overline{\eta}}}  
\ar[u]_-{r_{{s}}}
&  
 H^*(M{\rm{Hodge}}^{ss}_{ C_B})    
 \ar[l]_-{\rho_{0_B}}
 \ar[r]^-{\rho_{1_B}} 
 \ar[d]^-{r_{\overline{\eta}}} 
 \ar[u]_-{r_{s}}
& 
H^*(M{\rm de \, Rham}^{ss}_{ C_B})
\ar[d]^-{r_{\overline{\eta}}}  
\ar[u]_-{r_{{s}}}
\\
H^*(M{\rm{Higgs}}^{ss}_{ C_{\overline{\eta}}})   
\ar@/_2pc/[rr]_-{sp_{\tau_{\overline{\eta}}}}^-?
&
H^*(M{\rm{Hodge}}^{ss}_{ C_{\overline{\eta}}})  
\ar[l]_-{\rho_{0_{\overline{\eta}}}}
\ar[r]^-{\rho_{1_{\overline{\eta}}}}
& 
H^*(M{\rm de \, Rham}^{ss}_{ C_{\overline{\eta}}}) 
,
}
\end{equation}
where we have omitted indicating the possible specialization arrow  ${sp_v}_{{\rm{Hodge}},B}$ in the central column for graphical reasons, and the specialization arrows
 are labeled by a ``?" because at this stage we do not know whether they are defined.
 
 \cite[Prop. 3.3. (ii)]{decataldo-zhang-nahpostive} proves that the filtered specialization morphism
${sp_v}_{{\rm{Higgs}},B}$ exists and is an isomorphism in the case without poles under suitable coprimality conditions.
The same principle of proof applies here. \cite{decataldo-zhang-nahpostive} does not address the similar question
in the de Rham case, nor in the Hodge case.   

Theorem \ref{thm:spv} puts a remedy to these omissions.
We show that  the specialization morphisms of type $sp_{v_B}$ (\ref{eq:3vb}) exist and are isomorphisms
and that, moreover, they are compatible with the specialization morphisms
of type $sp_{\tau{s}}$ and  $sp_{\tau_{\overline{\eta}}}$.
For technical reasons, if $B$ is of mixed characteristic $(0,p>0),$ then we need to assume that $p>n.$
In the case of equal positive characteristic $p,$ we prove, without the need to assume $p>n,$ that  this system of specialization morphisms
of type $sp_{\tau_{s, \overline{\eta}}}$ and $sp_{{v}_B}$ are also  compatible with the
perverse filtrations.

\begin{thm}\label{thm:spv}
Let $(B,s,\overline{\eta})$ be a strictly Henselian DVR. Assume that the rank $n$ and degree $d$ are coprime, and the fibers of the divisor $D$ are nonempty.
If $B$ has mixed characteristic $(0,p>0),$ then, in addition, we assume that  $p>n.$

The specialization morphisms $sp_{v_{{\rm{Hodge}},B}}, sp_{v_{{\rm{Higgs}},B}}, sp_{v_{{\rm de \, Rham},B}}$
are defined, are isomorphisms, and they are compatible with the specialization morphisms
$sp_{\tau_{s}}$ and $sp_{\tau_{\overline{\eta}}}$
of Theorem \ref{thm:spk},
 i.e. we have the natural commutative diagram of restrictions/pull-backs and specializations,
 all of which are isomorphisms of cohomology rings: ($sp_{v_{{\rm{Hodge}},B}}$ is omitted for graphical reasons)
 \begin{equation}\label{iso55}
\xymatrix{
H^*(M{\rm{Higgs}}^{ss}_{ C_{s}}) 
\ar@/^2pc/[rr]^-{sp_{\tau_{s}}}_-{\simeq}
\ar@/_6pc/[dd]_-{sp_{\tau_{v_{{\rm{Higgs}},B}}}}^-{\simeq}
&  
H^*(M{\rm{Hodge}}^{ss}_{ C_{s}})   
\ar[l]_-{\rho_{0_{s}}}^-{\simeq} 
\ar[r]_-{\simeq}^-{\rho_{1_{s}}}
& 
H^*(M{\rm de \, Rham}^{ss}_{ C_{s}}) 
\ar@/^6pc/[dd]^-{sp_{\tau_{v_{{\rm de \, Rham},B}}}}_-{\simeq}  
\\
H^*(M{\rm{Higgs}}^{ss}_{ C_B) } 
\ar[d]^-{r_{\overline{\eta}}}_-{\simeq}  
\ar[u]^-{\simeq}_-{r_{{s}}}
&  
 H^*(M{\rm{Hodge}}^{ss}_{ C_B})    
 \ar[l]_-{\rho_{0_B}}^-{\simeq}
 \ar[r]_-{\simeq}^-{\rho_{1_B}} 
 \ar[d]^-{r_{\overline{\eta}}}_-{\simeq}  
 \ar[u]^-{\simeq}_-{r_{s}}
& 
H^*(M{\rm de \, Rham}^{ss}_{ C_B})
\ar[d]^-{r_{\overline{\eta}}}_-{\simeq}  
\ar[u]^-{\simeq}_-{r_{{s}}}
\\
H^*(M{\rm{Higgs}}^{ss}_{ C_{\overline{\eta}}})   
\ar@/_2pc/[rr]_-{sp_{\tau_{\overline{\eta}}}}^-{\simeq}
&
H^*(M{\rm{Hodge}}^{ss}_{ C_{\overline{\eta}}})  
\ar[l]_-{\rho_{0_{\overline{\eta}}}}^-{\simeq}  
\ar[r]_-{\simeq}^-{\rho_{1_{\overline{\eta}}}}
& 
H^*(M{\rm de \, Rham}^{ss}_{ C_{\overline{\eta}}}) 
;
}
\end{equation}
In  particular, we have the identity:
\[
 sp_{\tau_{\overline{\eta}}} \circ sp_{v_{{\rm{Higgs}}, B}}
 =    sp_{v_{{\rm de \, Rham}, B}} \circ sp_{\tau_{s}}.
\]

The vertical morphisms on the left-hand-side  Higgs column  in  (\ref{iso55}) are filtered isomorphisms for the respective perverse filtrations as in \S\ref{subsec:specmor}.

If $B$ has equal characteristic $p>0,$ then all the morphisms in (\ref{iso55}) are filtered isomorphisms
for the respective perverse filtrations  as in \S\ref{subsec:specmor}.
\end{thm}
\end{subsection}
\end{section}

\begin{section}{Deformation theory of $t$-connections} \label{sect:defth}
In this section, we construct an obstruction module and an
obstruction class to lifting $t$-connections for a square-zero thickening. We  also prove some compatibilities of the obstruction module with base-change.
These results are then used in Section \ref{subs:pfsmooth} to prove the smoothness assertion in Theorem \ref{thm:smooth}.

\begin{subsection}{Cech cohomology and base-change} \label{subs: cech}
Let $A$ be a Noetherian ring and let $Spec(A) \to \mathbb{A}^1_B$ be a morphism. 
If $M$ is an $A$-module, with associated 
$\mathcal{O}_{Spec(A)}$-module $\mathcal{M},$ and $\mathcal E$ is an $\mathcal{O}_{C_A}$-module, then
we denote $\mathcal{E} \otimes_{{\mathcal{O}_{C_A}}}
\pi_A^* \mathcal{M}$ simply by $\mathcal{E} \otimes_A M.$

Let $\mathcal{E}$ 
be a coherent $\mathcal{O}_{C_{A}}$-module. 
Let $\mathcal{U} = (U_i)_{i=1}^m$ be a finite affine open cover of the curve $C_A$. We have the corresponding \v{C}ech complex
 $(\widecheck{C}^{\bullet}(\mathcal{U},\mathcal{E}), \delta)$, whose definition we briefly recall next.
We set $U_{i_0, i_2, \ldots, i_l} := U_{i_0} \cap U_{i_1} \cap \ldots \cap U_{i_l}$, and  $\mathcal{E}_{i_0, i_1, \ldots, i_l} := \mathcal{E}(U_{i_0, i_1, \ldots, i_l})$. The (alternating) \v{C}ech complex has $l^{th}$-term given by
$\widecheck{C}^{l}(\mathcal{U},\mathcal{E})
 : = \prod_{(i_0< i_1< \ldots< i_l)} \mathcal{E}_{i_0, i_1, \ldots, i_l}$. The differential $\delta^l$ is given by: 
\[ (\delta^l (c) )_{i_0< \ldots <i_{l+1}} := \sum_{j=0}^{l+1} (-1)^{j+1} c_{i_0< i_1< \ldots, i_{j-1}< \widehat{i_{j}}< i_{j+1}< \ldots < i_{l+1}}.
\]

The following facts are standard, except possibly for part (d), where the morphism is ${\mathcal O}_A$-linear, but not $\mathcal{O}_{C_{A}}$-linear.
\begin{lemma} \label{lemma: base change cech cohomology lemma}
With notation as above, the following hold:
\begin{enumerate}[(a)]
    \item The $l^{th}$ cohomology group $\widecheck{H}^{l} (\mathcal{U}, \mathcal{E})$ of the \v{C}ech complex  computes the sheaf cohomology $H^l(\mathcal{E})=H^l(C_A,\mathcal{E})$. The \v{C}ech cohomology groups
    $\widecheck{H}^{l} (\mathcal{U}, \mathcal{E})$
    vanish for $l \geq 2$.
    \item Suppose that $\mathcal{E}$ is $A$-flat. Then, for any $A$-module $M$, the natural morphism:
    \[
    \xymatrix{
    H^1(\mathcal{E}) \otimes_{A} M \cong 
    \widecheck{H}^{l} (\mathcal{U}, \mathcal{E})
    \otimes_{A} M \ar[r] 
    &
    \widecheck{H}^{1} (\mathcal{U}, \mathcal{E}
    \otimes_{A} M) \cong H^1(\mathcal{E} \otimes_{A} M)
    }
    \]
    is an isomorphism,

    \item Suppose that $\mathcal{E}$ is $A$-flat. Let $S$ be an $A$-algebra, inducing a morphism $\sigma: Spec(S) \to Spec(A)$. 
    Then the natural morphism $H^1(\mathcal{E}) \otimes_{A} S \to H^1(\sigma_C^*(\mathcal{E}))$ is an isomorphism.
    \item With notation as in part (c), assume that $\mathcal{G}$ is another $A$-flat coherent sheaf 
    on $C_{A}$ equipped with a $A$-linear morphism of abelian sheaves $\varphi: \mathcal{E} \to \mathcal{G}$. Then, the  following 
    induced diagram is commutative:
    	\begin{figure}[H]
\centering
\begin{tikzcd}
  H^1(\mathcal{E})\otimes_{A} S \; \; \; \; \ar[r,"H^1(\varphi)\otimes_{A} S"] \ar[d, symbol = \xrightarrow{\sim}] & \; \; \; \;  H^1(\mathcal{G})\otimes_{A} S \ar[d, symbol = \xrightarrow{\sim}] \\  H^1((\sigma_C)^* \mathcal{E}) \; \; \ar[r, "H^1(\varphi \otimes_{A} S)"] & \; \; H^1((\sigma_C)^* \mathcal{G}).
\end{tikzcd}
\end{figure}
 
\end{enumerate}
\end{lemma}
\begin{proof}
\quad
\begin{enumerate}[(a)]
    \item 
    The higher cohomology of the coherent sheaf $\mathcal{E}$ over each of the affine subsets $U_{i_0, \ldots, i_l}$ vanishes. 
    Therefore, the \v{C}ech to sheaf cohomology spectral sequence
    (cf. 
    \cite[\href{https://stacks.math.columbia.edu/tag/03OU}{Tag 03O}]{stacks-project})
    for $\mathcal{E}$ is $E_2$-degenerate, thus yielding a canonical identification 
    $\widecheck{H}^{l} (\mathcal{U},\mathcal{E}) \cong H^l(\mathcal{E})$. 
    The vanishing of $\widecheck{H}^{l} (\mathcal{U},\mathcal{E}) \cong H^l(\mathcal{E})$ for $l \geq 2$ follows 
    from the theorem on formal functions, and the fact that the fibers of the morphism $C_{A} \to A$ 
    have dimension $1$ (cf. \cite[\href{https://stacks.math.columbia.edu/tag/02V7}{Tag 02V7}]{stacks-project}).
    
    \item By virtue of the $A$-flatness of $\mathcal{E}$, it follows that all 
    of the terms of the \v{C}ech complex $\widecheck{C}^{\bullet}(\mathcal{U},\mathcal{E})$ are flat $A$-modules.
    (cf. \cite[\href{https://stacks.math.columbia.edu/tag/01U4}{Tag 01U4}]{stacks-project}).
    Consider the truncation $\tau_{\leq 1}(
    \widecheck{C}^{\bullet}(\mathcal{U},\mathcal{E})),
    $ given by the two-term complex:
    \[ \tau_{\leq 1}(\widecheck{C}^{\bullet}(\mathcal{U},\mathcal{E}))) 
    = \left[ \widecheck{C}^{0}(\mathcal{U},\mathcal{E})) \xrightarrow{\delta^0} \text{ker}(\delta^1) \right]. \]
    This truncation is quasi-isomorphic to $\widecheck{C}^{\bullet}(\mathcal{U},\mathcal{E}))$, 
    via the given inclusion, because the cohomology in degree $\geq 2$ vanishes by (a). Moreover, 
    since all the terms $\widecheck{C}^{l}(\mathcal{U},\mathcal{E}))$ are flat $A$-modules, 
    it follows that for any $A$-module $M$ we have 
    $\tau_{\leq 1}(\widecheck{C}^{\bullet}(\mathcal{U},\mathcal{E}) \otimes_{A} M) 
    \cong \tau_{\leq1}(\widecheck{C}^{\bullet}(\mathcal{U},\mathcal{E})) \otimes_{A} M$. 
    Note that the \v{C}ech complex of $\mathcal{E} \otimes_{A} M$ coincides with $\widecheck{C}^{\bullet}(\mathcal{U},\mathcal{E}) \otimes_{A} M$. 
    By using this fact, we see that  $\tau_{\leq 1}(\widecheck{C}^{l} 
    (\mathcal{U}, \mathcal{E}\otimes_{A} M)) \cong \tau_{\leq 1}(\widecheck{C}^{\bullet}(\mathcal{U},\mathcal{E})) \otimes_{A} M$. 
    Therefore, it follows that
    \begin{align*}
        \widecheck{H}^1(\mathcal{E}\otimes_{A}M) 
        & \cong H^1(\tau_{\leq 1}(\widecheck{C}^{\bullet}(\mathcal{U},\mathcal{E} \otimes_{A} M)) )\\
        & \cong 
        H^1 \tau_{\leq 1}(\widecheck{C}^{\bullet}(\mathcal{U},\mathcal{E}) \otimes_{A} M  )) \\
     & = \text{coker}\left[ 
     \widecheck{C}^{0}(\mathcal{U},\mathcal{E})\otimes_{A} M \xrightarrow{\delta^0 \otimes_{A} 
     {\text id}_M} \text{ker}(\delta^1) \otimes_{A} M \right].
    \end{align*} 
    On the other hand, since, irrespective of flatness,  the operation $(-)\otimes_{A}M$
    commute with taking cokernels, we have that:
    \begin{align*}
        \text{coker}\left[ 
         \widecheck{C}^{0}(\mathcal{U},\mathcal{E})
        \otimes_{A} M \xrightarrow{\delta^0 \otimes_{A} \text{Id}_M} \text{ker}(\delta^1) \otimes_{A}M \right] & \cong \text{coker}\left[ 
        \widecheck{C}^{0}(\mathcal{U},\mathcal{E})
        \xrightarrow{\delta^0} \text{ker}(\delta^1) \right] \otimes_{A} M \\& \cong H^1(\tau_{\leq 1}(
        \widecheck{C}^{\bullet}(\mathcal{U},\mathcal{E})
        )) 
        \otimes_{A} M \\ &\cong \widecheck{H}^1(\mathcal{E}) \otimes_{A} M,
    \end{align*} 
    as predicated.
    \item This follows immediately from part (b) by setting $M = S$.
    \item By restricting $\varphi: \mathcal{E} \to \mathcal{G}$ to each $U_{i_0, \ldots, i_l}$, we see that $\varphi$ induces 
    an $A$-linear morphism of \v{C}ech complexes $\widecheck{\varphi}: 
    \widecheck{C}^{\bullet}(\mathcal{U},\mathcal{E})
    \to \widecheck{C}^{\bullet}(\mathcal{U},\mathcal{
    G})$. It is readily seen that the morphism $\widecheck{\varphi}\otimes_{A} \text{Id}_S$ is the morphism 
    corresponding to $\varphi \otimes_{A} \text{Id}_S$ for the corresponding \v{C}ech complexes of $\sigma_C^*\mathcal{E}$ 
    and $\sigma_C^*(\mathcal{G})$ (notice that here we are using that $\varphi$ is $A$-linear to form the tensor product!). 
    The commutativity of the diagram follows, by using the identifications provided by the previous part (c).
\end{enumerate}
\end{proof}

Fix a morphism $x_A: Spec(A) \to \mathcal{M}{\rm{Hodge}}_{C_B}^{ss}$ over $\mathbb{A}^1_B$. 
This amounts to a pair $(\mathcal{F}, \nabla)$ consisting of a vector bundle $\mathcal{F}$ on $C_{A}$ and a logarithmic $t_{A}$-connection $\nabla$. 

\begin{notn}
We
denote by:
\begin{equation}\label{fgt}
\xymatrix{
\varphi_{x_A}: \mathcal{E}nd(\mathcal{F})
\ar[r]
&
\mathcal{E}nd(\mathcal{F}) \otimes_{\mathcal{O}_{C_A}} \omega_{C_A/A}(D_A)
}
\end{equation}
the $\mathcal{O}_A$-linear morphism that sends a local section 
$\theta$ in $\mathcal{E}nd(\mathcal{F})$ to the commutator $\nabla \circ \theta - \theta \circ \nabla$.
\end{notn}

\begin{defn} \label{defn: obstruction module}
The module of obstructions $\mathcal{Q}_{x_A}$ is defined to be the $A$-module  cokernel of the following $A$-linear morphism in sheaf cohomology: 
\[
\mathcal{Q}_{x_A} : = \text{coker}\left[ H^1(\mathcal{E}nd(\mathcal{F}) \xrightarrow{H^1(\varphi_{x_A})} H^1( \mathcal{E}nd(\mathcal{F}) \otimes_{\mathcal{O}_{C_A}} \omega_{C_A/A}(D_A)) \right].
\]
Note that since $C_{A} \to A$ is proper, $\mathcal{Q}_{x_A}$ is a finitely generated $A$-module.
\end{defn}

Let $S$ be a Noetherian  ring  with the structure of an $A$-algebra, thus inducing a morphism  $\sigma: Spec(S) \to Spec(A)$. We shall denote by $x_S = (\sigma_{C}^*(\mathcal{F}), \sigma_C^*(\nabla))$ the $t_{S}$-connection on $C_{S}$ obtained by pulling back $x_A$ via $\sigma$ (using $A$-linearity as in \cite[Lemma 2.7]{Simpson-repnI}). More concretely, $\sigma_C^*(\nabla)$ is the unique connection that satisfies $\sigma_{C}^*(\nabla)(\sigma_C^*(s)) = \sigma_C^*(\nabla(s))$ for every local pullback section $\sigma_S^*(s)$ of the sheaf $\sigma_C^*(\mathcal{F})$ (this uniquely determines the morphism on all sections by the $t_{S}$-Leibniz rule).
\begin{coroll} \label{coroll: base change obstruction module}
There is a natural isomorphism
of $S$-modules $\mathcal{Q}_{x_S} \cong \mathcal{Q}_{x_A} \otimes_{A} S$ of $S$-modules, i.e. the formation of the obstruction module commutes with base change.
\end{coroll}
\begin{proof}
This follows from Lemma \ref{lemma: base change cech cohomology lemma}.(d), the definition of $\mathcal{Q}_{x_A}$, and the fact that the formation of cokernels commutes with tensor products.
\end{proof}
\end{subsection}

\begin{subsection}{Obstruction classes}\label{sec:obs-cl}
\begin{context}
Let $A$ be a local Artin algebra. Fix a morphism $Spec(A) \to \mathbb{A}^1_B$ and let $t_A \in A$
be the image of $t.$
For ease of notation, we employ simplified notation such as $\omega_{C_A/A}$
and $D_A.$
Choose a pair $x_A = (\mathcal{F}, \nabla),$
given by a $t_A$-connection on $C_A$ and 
corresponding to a morphism $x_A: Spec(A) \to \mathcal{M}{\rm{Hodge}}_{C_B}^{ss}$ over $\mathbb{A}^1_B.$ 
\end{context}

Let $\widetilde{A}$ be another local Artin algebra, equipped with a surjective local homomorphism $\widetilde{A} \twoheadrightarrow A$. Denote by $\iota: Spec(A) \to Spec(\widetilde{A})$ the resulting closed embedding.
We assume that the kernel $\widetilde{I}$ of this surjection is a square-zero ideal in $\widetilde{A}.$ 
Then $\widetilde{I}$ carries
a canonical structure of an $A$-module, denoted by $I.$

\begin{remark}\label{reminderI}
Let $\widetilde{M}$ be an $\widetilde{A}$-module.
We have the $A$-module $M:= \widetilde{M}/\widetilde{I} \widetilde{M}.$ Since $\widetilde{I}^2=0,$ there is no conflict 
with the notation we have chosen for $I.$ We have that $I=\widetilde{I} \otimes_{\widetilde{A}} A.$
Irrespective of $\widetilde{I}$ squaring to zero, we have that:
$M = \widetilde{M} \otimes_{\widetilde{A}} A;$
if $\widetilde{M}$ is $\widetilde{A}$-flat, then the natural
surjective $\widetilde{A}$-morphism $\widetilde{I} \otimes_{\widetilde{A}} \widetilde{M}
\twoheadrightarrow \widetilde{I}\widetilde{M}$ is an isomorphism.
We have  a canonical isomorphism of $A$-modules:
$\widetilde{I} \otimes_{\widetilde{A}} \widetilde{M}= I \otimes_A M;$
if in addition $\widetilde{M}$ is $\widetilde{A}$-flat, then
these two $A$-modules are also $A$-isomorphic to $\widetilde{I} \widetilde{M}.$ 
We also have the analogous relations for $\mathcal{O}_A$ and $\mathcal{O}_{\widetilde{A}}$-modules
respectively on $C_A$ and $C_{\widetilde{A}},$ respectively.
For example, we have that if $\widetilde{\mathcal{E}}$ is a locally
free $\mathcal{O}_{C_{\widetilde{A}}}$-module on 
$C_{\widetilde{A}},$
with restriction $\mathcal{E}$ to $C_A$, then we have a canonical isomorphism of $\mathcal{O}_{C_{A}}$-modules:
\begin{equation}\label{kloi}
\widetilde{\mathcal{E}} \otimes_{\mathcal{O}_{C_{\widetilde{A}}}}
\widetilde{I} = 
\mathcal{E} \otimes_{\mathcal{O}_{C_A}}
I.
\end{equation}
\end{remark}

 Choose a compatible morphism $Spec(\widetilde{A}) \to \mathbb{A}^1_{B}$, so that $\iota: Spec(A) \hookrightarrow Spec(\widetilde{A})$ is a morphism over $\mathbb{A}^1_B$. This gives a well-defined  lift $t_{\widetilde{A}}$
of $t_A$. We thus have a commutative diagram of solid arrows:
	\begin{figure}[H]\label{figh}
\centering
\begin{tikzcd}
  Spec(A) \ar[r,"x_A"] \ar[d, symbol = \hookrightarrow, "\iota", labels= left] &  \mathcal{M}{\rm{Hodge}}_{C_B}^{ss} \ar[d] \\ Spec(\widetilde{A}) \ar[r] \ar[ur, dashrightarrow,"y_{\widetilde{A}}"] & \mathbb{A}^1_B.
\end{tikzcd}
\end{figure}
We are interested in finding lifts as in the dotted arrow. This amounts to finding a $t_{\widetilde{A}}$-connection $y_{\widetilde{A}} = (\widetilde{\mathcal{F}}, \widetilde{\nabla})$ over $C_{\widetilde{A}}$ such that the pullback $\iota^*(y_{\widetilde{A}})$ is isomorphic to $x_A$.

The following proposition is key to the proof of the smoothness Theorem \ref{thm:smooth}. While it is probably standard, we could not locate a reference in the literature.
\begin{prop} \label{prop: obstruction class}
With notation as above, there exists a well-defined element $\text{ob}_{x_A} \in \mathcal{Q}_{x_A}\otimes_{A}I$ such that $\text{ob}_{x_A} = 0$ if and only if a lift $y_{\widetilde{A}}$ 
of $x_A$  exists. In particular such a lift $y_{\widetilde{A}}$ always exists if $\mathcal{Q}_{x_A} = 0$.
\end{prop}

In order to prove the proposition, we will make use of the following consequence of the nilpotent version of Nakayama's lemma.
\begin{lemma} \label{lemma: nilpotent nakayama}
Let $Spec(R)$ be an affine scheme, and let $J$ be a nilpotent ideal in $R$. Let $M$ be a locally-free $R$-module. Then, any trivialization of the vector bundle $M/JM$ on $Spec(R/J)$ lifts to a trivialization of $M$.
\end{lemma}
\begin{proof}
A trivialization $\overline{\psi}: (R/J)^{\oplus n} \xrightarrow{\sim} M/JM$ amounts to a choice of $n$ independent elements $\overline{m_1}, \overline{m_2}, \ldots \overline{m}_n \in M/JM$. For each $i$, fix the choice of a lift $m_i \in M$ mapping to $\overline{m}_i$ under the surjection $M \twoheadrightarrow M/JM$. We claim that the corresponding morphism $\psi: R^{\oplus n} \xrightarrow{\oplus_i m_i} M$ is an isomorphism, thus concluding the proof of the lemma.

Surjectivity follows by the nilpotent version of Nakayama's lemma  \cite[\href{https://stacks.math.columbia.edu/tag/00DV}{Tag 00DV (11)}]{stacks-project}.
Next, we prove injectivity.
Since $M$ is a projective $R$-module, the surjective morphism $R^{\oplus n} \twoheadrightarrow M$ splits as $R^{\oplus n} \cong M \oplus \text{ker}(\psi)$. Since the reduction
modulo $J$ $\overline{\psi}$  of $\psi$ is injective, we must then have that $\text{ker}(\psi)/J\text{ker}(\psi) = 0$. It follows by  \cite[\href{https://stacks.math.columbia.edu/tag/00DV}{Tag 00DV (9)}]{stacks-project} that $\text{ker}(\psi) =0$, as desired. 
\end{proof}

\begin{proof}[Proof of Proposition \ref{prop: obstruction class}]
 The closed embedding $\iota: C_A \to C_{\widetilde{A}}$ induces a homeomorphism on the underlying topological spaces. In particular, an open cover $\mathcal{U}$ of $C_A$ induces an evident open cover $\mathcal{U}_{\widetilde{A}}$ of $C_{\widetilde{A}}$,
 compatibly with restrictions, i.e. 
$\mathcal{U}_{\widetilde{A}}$ restricts to 
$\mathcal{U}_A.$ 

Fix a finite affine open cover  $\mathcal{U}= (U_{A;i})_{i=1}^m$  of $C_A$
on which the restriction of  $\mathcal F$
is trivializable. We employ the notation $U_{A;i_0, \ldots, i_l} = 
U_{A;i_0} \cap \ldots \cap  U_{A;i_l},$ and similarly
for 
$U_{\widetilde{A};i_0, \ldots, i_l}.$
By Lemma \ref{lemma: base change cech cohomology lemma}, the \v{C}ech cohomology with respect to these covers computes the corresponding sheaf cohomology groups. We do not make a distinction between \v{C}ech and sheaf cohomology, and we use the results in Lemma \ref{lemma: base change cech cohomology lemma} freely without further mention.

By standard deformation theory for vector bundles \cite[Thm. 7.1]{hartshorne-deformation}, the obstruction to lifting the vector bundle $\mathcal{F}$
 from $C_A$ to a vector bundle $\widetilde{\mathcal{F}}$ on  $C_{\widetilde{A}}$ lives in the second cohomology group $H^2(C_A,\mathcal{E}nd(\mathcal{F})).$ This group is $\{0\}$ because $C_A$ is a curve  over the affine
 $Spec(A)$ (cf. Lemma \ref{lemma: base change cech cohomology lemma}(a)). 
 
Choose a  locally free lift $\widetilde{\mathcal{F}}$ of $\mathcal{F}$ to $C_{\widetilde{A}}$. We use the notation $\mathcal{F}_{i} := \mathcal{F}|_{U_{A;i}}$ and $\nabla_i := \nabla|_{U_{A;i}},$ and analogously for restrictions to multiple intersection.
 Similarly, we set $\widetilde{\mathcal{F}}_i : = \widetilde{\mathcal{F}}|_{U_{\widetilde{A};i}}$.  
 
Choose trivializations of the $\mathcal{F}_i$ on the $U_{A;i}.$
 By virtue of Lemma \ref{lemma: nilpotent nakayama}, 
 we can and do choose trivializations of the  $\widetilde{\mathcal{F}}_i$ 
 that restrict to the chosen trivializations of
  the $\mathcal{F}_i.$ 
  
Since the open sets of the covers are affine, we can and do choose  lifts
of the $t_{A}$-connections $\nabla_i$ on the ${\mathcal F}_i$ to  $t_{\widetilde{A}}$-connections $\widetilde{\nabla}_i$ on the $\widetilde{\mathcal{F}}_i.$
 Indeed, under the trivialization of $\mathcal{F}_i$, the $t_{A}$-connection $\nabla_i$ can be written as $t_{A}d + M_i$ for some matrix  $M_i$ with entries sections
 of $\omega_{U_{A;i}}(D_A).$ Here, recalling that  
 $D_A:=D_B|_A$ is the pull-back of $D_B$ to $C_A,$ we have abbreviated $D_A|_{U_{A;i}}$ to $D_A,$
 and we have abbreviated
 $\omega_{C_A/A}|_{U_{A;i}}$ to
 $\omega_{U_{A;i}}.$ Note that,
 by the invariance properties of K\"ahler differentials, we have that
 $\omega_{U_{\widetilde{A};i}}|_{U_{A;i}}= \omega_{U_{A;i}}.$

 Since $U_i$ is affine and $\widetilde{A} \twoheadrightarrow A$ is surjective, we can lift $M_i$ to a matrix $\widetilde{M}_i$ 
 with entries  sections
 of $\omega_{U_{\widetilde{A};i}}(D_{\widetilde{A}}).$
This matrix $\widetilde{M}_i$ can be used to define the connection $\widetilde{\nabla}_i:= t_{\widetilde{A}}d + \widetilde{M}_i,$
a lift of $\nabla_i$ to the trivialized vector bundle $\widetilde{\mathcal{F}}_i$.
We define:
\begin{equation}\label{cocycle c}
c_{i,j} := \widetilde{\nabla}_i|_{U_{A;i,j}} - \widetilde{\nabla}_j|_{U_{A;i,j}}.
\end{equation}
The differences $c_{i,j}$ of these connections are sections  inside  $\Gamma (\mathcal{E}nd(\widetilde{\mathcal{F}}_{i,j}) \otimes \omega_{U_{\widetilde{A};i,j}}(D_{\widetilde{A}}))$. Since $\widetilde{\nabla}_i|_{U_{A;i,j}} = \widetilde{\nabla}_j|_{U_{A;i,j}}  = \nabla|_{U_{A;i,j}}$, the elements $c_{i,j}$ actually land
in the submodule 
$\widetilde{I}\cdot \Gamma (\mathcal{E}nd(\widetilde{\mathcal{F}}_{i,j}) \otimes \omega_{U_{\widetilde{A};i,j}}(D_{\widetilde{A}}))$. We will use the series of identifications
\begin{align*}
    \widetilde{I}\cdot \Gamma (\mathcal{E}nd(\widetilde{\mathcal{F}}_{i,j}) \otimes \omega_{U_{\widetilde{A};i,j}}(D_{\widetilde{A}})) & =   \Gamma (\widetilde{I}\cdot [\mathcal{E}nd(\widetilde{\mathcal{F}}_{i,j}) \otimes \omega_{U_{\widetilde{A};i,j}}(D_{\widetilde{A}})]) \\
& =\Gamma ([\mathcal{E}nd(\widetilde{\mathcal{F}}_{i,j} ) \otimes \omega_{U_{\widetilde{A};i,j}}(D_{\widetilde{A}})]
\otimes_{\widetilde{A}} \widetilde{I}) \\ & = \Gamma(\mathcal{E}nd(\mathcal{F}_{i,j}) \otimes
\omega_{U_{A;i,j}}(D_A) \otimes_{A}I).
\end{align*}
%which is naturally isomorphic to $\Gamma (\mathcal{E}nd(\widetilde{\mathcal{F}}^0_{i,j}) \otimes \omega_{U_{\widetilde{A};i,j}}(D_{\widetilde{A}}))\otimes_{\widetilde{A}} I$ by $\widetilde{A}$-flatness of $\Gamma (\mathcal{E}nd(\widetilde{\mathcal{F}}^0_{i,j}) \otimes \omega_{U_{\widetilde{A};i,j}}(D_{\widetilde{A}}))$. 
Here the second identification follows from the flatness of the argument in square parentheses and the last identification follows from \ref{kloi}). 
%and (\ref{5rg}), i.e. under the natural isomorphism $\mathcal{E}nd(\mathcal{F}_{i,j}) \otimes \omega_{U_{A;i,j}}(D_A) \otimes_{A}I \cong \mathcal{E}nd(\mathcal{\widetilde{F}}_{i,j}) \otimes \omega_{U_{\widetilde{A};i,j}}(D_{\widetilde{A}}) \otimes_{\widetilde{A}}I.$

We can thus view  $c :=(c_{i,j})$ as a cochain in $\widecheck{C}^1(\mathcal{U}, \mathcal{E}nd(\mathcal{F}) \otimes \omega_{C_A/A}(D_A) \otimes_{A} I)$.  By its very definition,
this cochain is a cocyle, i.e. $\delta^1(c)_{i,j,k} =0.$
%\[ \delta^1(c)_{i,j,k} = ( \widetilde{\nabla}_i|_{U_{A;i,j,k}} - \widetilde{\nabla}_j|_{U_{A;i,j,k}}) - (\widetilde{\nabla}_i|_{U_{A;i,j,k}} - \widetilde{\nabla}_k|_{U_{A;i,j,k}}) + (\widetilde{\nabla}_j|_{U_{A;i,j,k}} - \widetilde{\nabla}_k|_{U_{A;i,j,k}}) = 0, \]
We denote by $[c]$ the corresponding cohomology class in $\widecheck{H}^1 (\mathcal{U},\mathcal{E}nd(\mathcal{F}) \otimes \omega_{C_A/A}(D_A) \otimes_{A} I) \cong$ (cf.
Lemma \ref{lemma: base change cech cohomology lemma}(b))
$\widecheck{H}^1 (\mathcal{U},\mathcal{E}nd(\mathcal{F}) \otimes \omega_{C_A/A}(D_A)) \otimes_A I.$

 Let $\text{ob}_{x_A}$ to be the image $\overline{[c]}$ of $[c]$ in
 $\mathcal{Q}_{x_A} \otimes_{A} I$, the obstruction module
 twisted by $I$ (cf. (\ref{defn: obstruction module}).
 In order to prove the proposition, 
 we need to check that: i) $\text{ob}_{x_A}$ does not depend on the choices involved in its construction, namely
 the lift $\widetilde{\mathcal{F}}$ of $\mathcal{F},$ the
 trivializing cover $\mathcal{U}$  and  the lifts $\widetilde{\nabla}_i$ of the $\nabla_i$; ii) $\text{ob}_{x_A}$ vanishes if and only if there exists a lift $(\widetilde{\mathcal{F}}, \widetilde{\nabla})$ to $C_{\widetilde{A}}$ of $(\mathcal{F}, \nabla)$ on $C_A.$

Fix a lift $\widetilde{\mathcal{F}}$ and a trivializing cover $\mathcal{U}.$ Choose two different sets of lifts $\widetilde{\nabla}_i$ and $\widetilde{\nabla}'_i.$

We have the cocycle $c$ and the cocycle $c'=(c_{i,j}^1):= \widetilde{\nabla}'_i|_{U_{A;i,j}}- \widetilde{\nabla}'_j|_{U_{A;i,j}}.$ Since $\widetilde{\nabla}'_i \equiv \widetilde{\nabla}_i\,  (mod \, \widetilde{I})$, their common restriction to $U_{A;i}$ being $\nabla_i,$ the difference $h_i : = \widetilde{\nabla}'_i - \widetilde{\nabla}_i$ is an element of $\Gamma(\mathcal{E}nd(\mathcal{F}_i) \otimes \omega_{U_{A;i}} (D_A)) \otimes_{A} I$ by the same reasoning as above. We can view $h=(h_i)$ as a cochain in $\widecheck{C}^0(\mathcal{U},\mathcal{E}nd(\mathcal{F}) \otimes \omega_{C_A/A}(D_A)) \otimes_{A} I$. 
By construction, we have:
\[ 
c' = c + \delta^0(h),
\]
so that $[c'] = [c] \in  \widecheck{H}^1 (\mathcal{U},
\mathcal{E}nd(\mathcal{F}) \otimes \omega_{C_A/A}(D_A)) \otimes_{A} I.$

\textbf{N.B. I:}
Conversely, for any $0$-chain $h=(h_i)$ we can define $\widetilde{\nabla}'_i := \widetilde{\nabla}_i +h_i$, and then we end up with cohomologous cycles $c'= c + \delta^0(h)$.

Let $\mathcal{V}$ be a finite refinement of the  given trivializing cover $\mathcal{U}.$ Let $\tau, \tau': I_{\mathcal{V}} \to I_{\mathcal{U}}$
be any two refinement maps on the indexing sets of the covers, so that
$V_x \subseteq U_{\tau (x)} \cap U_{\tau' (x)}.$
We denote by the same symbol the induced chain homotopic morphisms of cochain complexes
$\tau, \tau': \widecheck{C} (\mathcal{U},-) \to \widecheck{C} (\mathcal{V},-)$
(cf. \cite[\href{https://stacks.math.columbia.edu/tag/09UY}{Tag 09UY}]{stacks-project}). A choice of lifts $\widetilde{\nabla}_i$ gives rise to the cocycle $c$ for $\mathcal{U}$ as above. The  two choices
of lifts
$\widetilde{\nabla}_{\tau (x)}| V_x$  and $\widetilde{\nabla}_{\tau' (x)}| V_x,$  give rise to
 corresponding cocycles for $\mathcal{V}$ denoted $\gamma$ and $\gamma'.$ By construction, we see that $\tau (c) = \gamma$ and $\tau'(c)=\gamma',$ so that the formation of the cohomology class $[c]$,
for a given lift $\widetilde{\mathcal{F}},$ is compatible with finite refinements of covers. The usual argument involving common  refinements tells us that the formation of the class $[c]$
depends only on the choice of lift $\widetilde{\mathcal{F}}.$
We are thus  left with showing that
given a trivializing cover, the obstruction class is independent of the choice
of lift $\widetilde{\mathcal{F}}.$

We choose a second locally free lift $\widetilde{\mathcal{F}}^1$ to $C_{\widetilde{A}}$ of the vector bundle $\mathcal{F}$ on $C_{\widetilde{A}}$. Since the restrictions $\mathcal{F}_i, \widetilde{\mathcal{F}}^1_i, \widetilde{\mathcal{F}}_i$ are trivializable on their corresponding affine schemes, we can choose
isomorphisms $\psi_i: \widetilde{\mathcal{F}}_i \xrightarrow{\sim} \widetilde{\mathcal{F}}^1_i$ that restrict to the identity on $\mathcal{F}_i$. For each pair of indexes $i,j$, the isomorphism ${\psi_j}^{-1}|U_{\widetilde{A};i,j} \circ \psi_i|U_{\widetilde{A};i,j}: \widetilde{\mathcal{F}}_{i,j} \to \widetilde{\mathcal{F}}_{i,j},$ is a lift of the identity,
so that it is of the form  $\text{Id} + B_{i,j}$ for a unique element $B_{i,j}\in \text{End}(\widetilde{\mathcal{F}}_{i,j}) \otimes_{\widetilde{A}} \widetilde{I} \cong \text{End}(\mathcal{F}_{i,j}) \otimes_{A} I$. We can view $B= (B_{i,j})$ as a cochain in $\widecheck{C}^1 (\mathcal{U}_A, (\mathcal{E}nd(\mathcal{F}) \otimes_A I)$. By direct computation, we see that $B$ is  a cocycle. We can use $\psi_i$ to define lifts $\widetilde{\nabla}^1_i = \psi_i \circ \widetilde{\nabla}_i \circ \psi^{-1}_i$ on $\widetilde{\mathcal{F}}^1_i$, so that they fit into the following commutative diagram:
\begin{figure}[H]
\centering
\begin{tikzcd}
  \widetilde{\mathcal{F}}_i \ar[r, "\widetilde{\nabla}_i"] \ar[d, "\psi_i", labels = left] &  \widetilde{\mathcal{F}}_i
  \otimes \omega_{C_{\widetilde{A}/\widetilde{A}}} (D_{\widetilde{A}})
  \ar[d, "\psi_i"] 
  \\ 
  \widetilde{\mathcal{F}}^1_i \ar[r, "\widetilde{\nabla}^1_i"] &  \widetilde{\mathcal{F}}^1_i
  \otimes \omega_{C_{\widetilde{A}/\widetilde{A}}} (D_{\widetilde{A}}).
\end{tikzcd}
\end{figure}
The new \v{C}ech cocycle $c^1= (c^1_{i,j})$ for this choice of $\widetilde{\mathcal{F}}^1$ and $\widetilde{\nabla}^1_i$ is given by:
\[ c^1_{i,j} =  \widetilde{\nabla}^1_i|_{U_{\widetilde{A};i,j}} - \widetilde{\nabla}^1_j|_{U_{\widetilde{A};i,j}} = (\psi_i \circ \widetilde{\nabla}_i \circ \psi^{-1}_i)|_{U_{\widetilde{A};i,j}} - (\psi_j \circ \widetilde{\nabla}_j \circ 
\psi^{-1}_j)|_{U_{\widetilde{A};i,j}}.
\]
We know that $c^1=(c_{i,j}^1)$ is actually a cocycle dwelling in the submodule 
$\Gamma( \mathcal{E}nd(\mathcal{F}_{i,j}) \otimes 
\omega_{C_A/A}(D_A)) \otimes_{A} I$. Since $\psi_j$ restricts to the identity on $\mathcal{F}_j$, we see that applying $\psi_j^{-1}|_{\widetilde{U}_{i,j}} \circ (-) \circ \psi_j|_{\widetilde{U}_{i,j}}$ does not affect the cocycle in $\Gamma( \mathcal{E}nd(\mathcal{F}_{i,j}) \otimes 
\omega_{C_A/A}(D_A)) \otimes_{A} I$. We can thus re-write:
\[ c_{i,j}^1 = \psi_j^{-1}|_{U_{\widetilde{A};i,j}} \circ c^1_{i,j} \circ \psi_j|_{U_{\widetilde{A};i,j}} = (\psi^{-1}_j \circ \psi_i \circ \widetilde{\nabla}_i \circ \psi^{-1}_i \circ \psi_j)|_{U_{\widetilde{A};i,j}} - \widetilde{\nabla}_{j}|_{U_{\widetilde{A};i,j}} \]
This can be re-written as:
\[ c_{i,j}^1 = (1+B_{i,j}) \circ \widetilde{\nabla}_i \circ (1-B_{i,j})|_{U_{\widetilde{A};i,j}} - \widetilde{\nabla}_{j}|_{U_{\widetilde{A};i,j}}. 
\]
By expanding, and using that $\text{End}(\mathcal{F}_{i,j}) \otimes_{A} I$
and $\widetilde{I}^2=0$, we get: (we omit denoting the restrictions
to ${U_{\widetilde{A};i,j}}$)
\[ c_{i,j}^1 = - \widetilde{\nabla}_i \circ B_{i,j}
+ B_{i,j} \circ \widetilde{\nabla}_i + \widetilde{\nabla}_{i}
- \widetilde{\nabla}_j  =  -
\widetilde{\nabla}_i \circ B_{i,j} + B_{i,j} \circ \widetilde{\nabla}_i + c_{i,j}.
\]
Now, since $B_{i,j}$ lies in the submodule $\text{End}(\mathcal{F}_{i,j}) \otimes_{A} I$ and $\widetilde{\nabla}_i$ is a lift of $\nabla_i$, the commutator can be rewritten as: (omitting restrictions again,
and recalling (\ref{fgt}))
\[
\widetilde{\nabla}_i \circ B_{i,j} - B_{i,j} \circ \widetilde{\nabla}_i = \nabla_i \circ B_{i,j}- B_{ij} \circ \nabla_i 
= (\widecheck{C}^1(\varphi_{x_A})(B))_{i,j}.
\]
%Here, we have written: \[\widecheck{C}^1(\varphi_{x_A}): \widecheck{C}^1 (\mathcal{U}_A,\mathcal{E}nd(\mathcal{F})) \otimes_{A} I \to \widecheck{C}^1 ({\mathcal U}_A, \mathcal{E}nd(\mathcal{F}) \otimes \omega_{C_A/A} (D_A)) \otimes_{A} I\] to denote the $A$-linear map induced by $\varphi_{x_A}: \mathcal{E}nd(\mathcal{F}) \to \mathcal{E}nd(\mathcal{F}) \otimes \omega_{C_A/A}(D_A)$ (defined just before Definition \ref{defn: obstruction module}) at the level of \v{C}ech chains. 
 In conclusion, the new cocycle $c^1$ can be expressed as
 \begin{equation} \label{equation: cocycle change by lift of vector bundle}
     c^1= -\widecheck{C}^1(\varphi_{x_A})(B) + c.
 \end{equation}
Hence it differs from the cocycle $c$ by the image of a cocycle in $\widecheck{C}^1 (\mathcal{U}_A,\mathcal{E}nd(\mathcal{F})) \otimes_{A} I$, and so,
 given Definition \ref{defn: obstruction module} of the obstruction $A$-module $\mathcal{Q}_{x_A}$, it yields the same element $\text{ob}_{x_A}$ in $\mathcal{Q}_{x_A} \otimes_{A} I,$. We have established the sought-after indepedence on the locally free lift $\widetilde{\mathcal{F}}$
 of $\mathcal{F}$.

\textbf{N.B. II:}
In the last argument, the cocycle $B=(B_{i,j})$ depends on the choice of isomorphisms $\psi_i$ up to the coboundary of a $0$-chain. Indeed, we can always change the isomorphisms $\psi_i$ by precomposing by an automorphism of $\widetilde{\mathcal{F}}|_{i}$, which will be of the form $\text{Id} + M_i$ for some cochain $M=(M_i)$ in $\widecheck{C}^0(\mathcal{U}, \mathcal{E}nd(\mathcal{F}) \otimes_{A}I)$. It follows from the computations above that the new cocycle obtained by changing the $\psi_i$ in this way will be of the form $B + \delta^0(M)$. We conclude that the corresponding cohomology class $\overline{B}$ of $B$ in $\widecheck{H}^1( \mathcal{U}, \mathcal{E}nd(\mathcal{F})) \otimes_{A} I$ is well-defined. Conversely, by standard deformation theory of vector bundles \cite[Thm. 7.1]{hartshorne-deformation}, every such cohomology class $\overline{B}$ arises this way from a choice of a locally free lift
of $\mathcal{F}$ to $C_{\widetilde{A}}.$ This establishes a canonical bijection
between isomorphism classes of lifts 
$\widetilde{\mathcal{F}}$ of $\mathcal{F}$ and cohomology classes in $\widecheck{H}^1 (\mathcal{U}, \mathcal{E}nd(\mathcal{F})) \otimes_{A} I$. 

Hence, for any given cocycle $B$ in $\widecheck{C}^1(\mathcal{U}, \mathcal{E}nd(\mathcal{F}) \otimes_{A} I)$, we can find a given lift $\widetilde{\mathcal{F}}^1$ and isomorphisms $\psi_i$ such that the corresponding cocycle is cohomologous to $B$. By further changing the given $\psi_i$ by a $0$-cocycle $M_i$ as described above, we can moreover assume that the corresponding cocycle in $\widecheck{C}^1 (\mathcal{U}_A, (\mathcal{E}nd(\mathcal{F}) \otimes_A I)$ is $B$ on the nose. Hence, for any given cocycle $B$ we find lifts as described above so that the new obstruction cocycle is $c^1= -\widecheck{C}^1(\varphi_{x_A})(B) + c$, as in the computation above (Equation \ref{equation: cocycle change by lift of vector bundle}).

We conclude by showing that $\text{ob}_{x_A}=0$ if and only if there exists a lift of the $t_{A}$-connection to $C_{\widetilde{A}}$. 

First, suppose that $\text{ob}_{x_A} = 0$. Choose a suitably finite trivializing cover $\mathcal{U}$ and some lifts $\widetilde{\nabla}_i$ and $\widetilde{\mathcal{F}}$. The corresponding cocycle  $c=(c_{i,j})$ satisfies $[\overline{c}]=0$, and so it is cohomologous to an element in the image of $\widecheck{C}^1(\varphi_{x_A})$. We can thus
find $h =(h_i) \in \widecheck{C}^0 (\mathcal{U}, \mathcal{E}nd(\mathcal{F}) \otimes \omega_{C_A/A}(D_A) \otimes_{A} I)$ and $B= (B_{i,j}) \in \widecheck{Z}^1 (\mathcal{U}, \mathcal{E}nd(\mathcal{F}) \otimes_{A} I),$ such that $c = \delta^0(h) + \widecheck{C}^1(\varphi_{x_A})(B)$. 
Replacing the lifts with $\widetilde{\nabla}_i := \nabla_i - h_i$ instead, as in \textbf{N.B. I}, we can assume that the cocycle $c$ is of the form $c= \widecheck{C}^1(\varphi_{x_A})(B)$. 
Now, by \textbf{N.B.2}, we can choose another lift $\widetilde{\mathcal{F}}^1$ and isomorphisms $\psi_i$ that correspond to the cocycle $B$. As in \textbf{N.B. II}, this yields new choices of lifts such that the corresponding new cocycle $c^1 = c -\widecheck{C}^1(\varphi_{x_A})(B)=0$ vanishes. 
Hence, we can assume without loss of generality that $c$ is identically $0$. 
Since by definition we have $c_{i,j}= \widetilde{\nabla}^1_i|_{\widetilde{U}_{i,j}}- \widetilde{\nabla}^1|_j|_{\widetilde{U}_{i,j}} $, this means that the $\widetilde{\nabla}_i^1$ agree on the intersections $U_{\widetilde{A};i,j}$, and so they glue to give a $t_{\widetilde{A}}$-connection $(\widetilde{\mathcal{F}},\widetilde{\nabla})$ that lifts $(\mathcal{F},\nabla).$

Conversely, suppose that there exists a lift $(\widetilde{\mathcal{F}}, \widetilde{\nabla})$ of the $t_{A}$-connection to $C_{\widetilde{A}}$. Then can use $\widetilde{\mathcal{F}}$ as the lift of the vector bundle, choose any trivializing cover, and set 
 $\widetilde{\nabla}_i := \widetilde{\nabla}_{|\widetilde{U}_i}$ in the construction of a cocycle $c=(c_{i,j})$ representing $\text{ob}_{x_A}$. Since the $\widetilde{\nabla}_i$ agree on the intersections, we have $c_{i,j} = 0$, so that $\text{ob}_{x_A} = 0$.
\end{proof}
\end{subsection}

\begin{subsection}{Relative tangent space}
For any algebraic stack $\mathcal{M}$ and any geometric point $x: \text{Spec}(k) \to \mathcal{M}$, the tangent space $T_{\mathcal{M}, x}$ is defined to be the set of isomorphism classes of pairs $(y, \psi)$, where $y: Spec(k[\epsilon]/(\epsilon^2)) \to \mathcal{M}$ is a $k[\epsilon]/(e^2)$-point of $\mathcal{M}$ and $\psi$ is an isomorphism $y|_{\text{Spec}(k[\epsilon]/(\epsilon))} \xrightarrow{\sim} x$. The tangent space $T_{\mathcal{M}, x}$ acquires a canonical structure of a $k$-vector space.

We shall describe the tangent spaces of the fibers of $\mathcal{M}{\rm{Hodge}}^{ss}_{C_B} \to \mathbb{A}^1_B$. Fix a geometric point $a: \text{Spec}(k) \to \mathbb{A}^1_{B}$, and choose a geometric point $x: \text{Spec}(k) \to (\mathcal{M}{\rm{Hodge}}_{C_B}^{ss})_a$ of the fiber $(\mathcal{M}{\rm{Hodge}}_{C_B}^{ss})_a$. The point $x$ represents a pair $(\mathcal{F}, \nabla)$ of a vector bundle and a logarithmic $t_a$-connection. In Subsection \ref{subs: cech} we made use of the following complex of sheaves of $k$-vector spaces to define the obstruction module
\[ C^{\bullet}(x) := \left[ \mathcal{E}nd(\mathcal{F}) \xrightarrow{\varphi_x} \mathcal{E}nd(\mathcal{F}) \otimes_{\mathcal{O}_{C_A}} \omega_{C_A/A}(D_A)\right]\]
Here by convention we place the left term $\mathcal{E}nd(\mathcal{F})$ in cohomological degree $0$. We shall denote by $\mathbb{H}^i\left(C^{\bullet}(x)\right) := \mathbb{H}^i\left(C_a, C^{\bullet}(x)\right)$ denote the $i^{th}$ hypercohomology of the complex. By the hypercohomology spectral sequence, we have a natural identification $\mathbb{H}^2\left(C^{\bullet}(x)\right) \cong \mathcal{Q}_x$. The spectral sequence also identifies $\mathbb{H}^0\left(C^{\bullet}(x)\right)$ with the $k$-vector space $\text{End}(x)$ consisting of endomorphisms of the vector bundle $\mathcal{F}$ that commute with the logarithmic $t_a$-connection $\nabla$. The argument in \cite[Thm. 4.2]{nitsure-log-connections} generalizes without change to the setting of logarithmic $t_a$-connections to show that there is a natural identification of $k$-vector spaces $\mathbb{H}^1\left(C^{\bullet}(x)\right) \cong T_{(\mathcal{M}{\rm{Hodge}}_{C_B}^{ss})_a, x}$ (see also \cite[\S5]{hao-lambda-modules-dm-stacks} for a treatment in the generality of $\Lambda$-modules). Using these identifications, we give a dimension formula for the tangent space $T_{(\mathcal{M}{\rm{Hodge}}_{C_B}^{ss})_a, x}$.
\begin{coroll} \label{coroll: relative tangent space}
Fix $a \in \mathbb{A}^1_B$. For any geometric point $x \in (\mathcal{M}{\rm{Hodge}}_{C_B}^{ss})_a$ of the fiber, the dimension of the tangent space $T_{(\mathcal{M}{\rm{Hodge}}_{C_B}^{ss})_a, x}$ of $x$ in $(\mathcal{M}{\rm{Hodge}}_{C_B}^{ss})_a$ is given by
\[ \text{dim}\left(T_{(\mathcal{M}{\rm{Hodge}}_{C_B}^{ss})_a, x}\right) = n^2(2g-2 + {\rm{deg}}(D_a)) + \text{dim}\left({\rm{End}}(x)\right) + \text{dim}(\mathcal{Q}_{x})\]
In particular, if the rank $n$ and degree $d$ are coprime, then the dimension of the tangent space of the fiber is given by
\[ \text{dim}\left(T_{(\mathcal{M}{\rm{Hodge}}_{C_B}^{ss})_a, x}\right) = n^2(2g-2 + {\rm{deg}}(D_a)) + 1 + \text{dim}(\mathcal{Q}_{x})\]
\end{coroll}
\begin{proof}
The point $x$ represents a logarithmic $t_a$-connection $(\mathcal{F}, \nabla)$. By the hypercohomology spectral sequence and Riemann-Roch, we get the following formula for the Euler characteristic
\[ \chi\left(C^{\bullet}(x)\right) = \chi(\mathcal{E}nd(\mathcal{F})) + \chi\left(\mathcal{E}nd(\mathcal{F}) \otimes \Omega^1_{C_a/a}(D_a)\right) = -n^2(2g-2 + \text{deg}(D_a)) \]
Since by definition $\chi\left(C^{\bullet}(x)\right) = \mathbb{H}^0\left(C^{\bullet}(x)\right) - \mathbb{H}^1\left(C^{\bullet}(x)\right) + \mathbb{H}^2(\left(C^{\bullet}(x)\right))$, we get
\[ \text{dim}\left(\mathbb{H}^1\left(C^{\bullet}(x)\right)\right) = n^2(2g-2 + \text{deg}(D_a)) + \text{dim}\left(\mathbb{H}^0\left(C^{\bullet}(x)\right)\right) + \text{dim}\left(\mathbb{H}^0\left(C^{\bullet}(x)\right)\right)\]
Using the natural identifications $\mathbb{H}^0\left(C^{\bullet}(x)\right) \cong \text{End}(x)$, $\mathbb{H}^1\left(C^{\bullet}(x)\right) \cong T_{(\mathcal{M}{\rm{Hodge}}_{C_B}^{ss})_a, x}$ and $\mathbb{H}^2\left(C^{\bullet}(x)\right) \cong \mathcal{Q}_x$ yields the desired formula
\[ \text{dim}\left(T_{(\mathcal{M}{\rm{Hodge}}_{C_B}^{ss})_a, x}\right) = n^2(2g-2 + {\rm{deg}}(D_a)) + \text{dim}\left({\rm{End}}(x)\right) + \text{dim}(\mathcal{Q}_{x})\]
In the special case when $n$ and $d$ are coprime, then the space of endomorphisms $\text{End}(x)$ is one dimensional, consisting of the scalar endomorphisms of $\mathcal{F}$ (cf. the proof of Lemma \ref{lemma: rigidification is alg space}). Hence we can set $\text{dim}\left({\rm{End}}(x)\right) =1$.
\end{proof}
\end{subsection}
\end{section}

\begin{section}{Smoothness and irreducibility of the moduli space}\label{sec:smirr}

 \begin{subsection}{Reduction to the smoothness of the stack}\label{subs:redsmst}

There is a central copy of $\mathbb{G}_m$ in the automorphisms of every point of $\mathcal{M}{\rm{Hodge}}_{C_B}^{ss}$, because multiplication by constants commutes with any logarithmic $t$-connection. Therefore, we can form the $\mathbb{G}_m$-rigidification $(\mathcal{M}{\rm{Hodge}}^{ss}_{})^{rig}$, as in
\cite[Appendix A]{abramovich2007tame}. By the proof of \cite[Thm. A.1]{abramovich2007tame}, there is a smooth cover $U \to (\mathcal{M}{\rm{Hodge}}^{ss}_{C_B})^{rig}$ by a scheme $U$ and a Cartesian diagram
\begin{figure}[H]
\centering
\begin{tikzcd}
  B(\mathbb{G}_{m,U}) \ar[r] \ar[d] & \mathcal{M}{\rm{Hodge}}^{ss}_{C_B} \ar[d] \\ U \ar[r] & (\mathcal{M}{\rm{Hodge}}^{ss}_{C_B})^{rig}
\end{tikzcd}
\end{figure}
Since the left vertical arrow $B(\mathbb{G}_{m,U}) \to U$ is a smooth good moduli space morphism, and being a good moduli space morphism can be checked \'etale locally on the target, it follows that the rigidification morphism $\mathcal{M}{\rm{Hodge}}^{ss}_{C_B} \to (\mathcal{M}{\rm{Hodge}}^{ss}_{C_B})^{rig}$ is a smooth good moduli space morphism. In particular, since being Noetherian can be checked smooth locally, it also follows that $(\mathcal{M}{\rm{Hodge}}^{ss}_{C_B})^{rig}$ is Noetherian.
\begin{lemma} \label{lemma: rigidification is alg space} Assume that the rank $n$ and degree $d$ are coprime. Then $(\mathcal{M}{\rm{Hodge}}^{ss}_{C_B})^{rig}$ is an algebraic space.
\end{lemma}
\begin{proof} 
We need to show that inertia is trivial \cite[\href{https://stacks.math.columbia.edu/tag/04SZ}{Tag 04SZ}]{stacks-project}. This means that for every $T$-point $p^{rig}: T \to (\mathcal{M}{\rm{Hodge}}^{ss}_{C_B})^{rig}$, we need to show that the group algebraic space of automorphisms $\text{Aut}(p^{rig}) \to T$ is trivial. Since $(\mathcal{M}{\rm{Hodge}}^{ss}_{C_B})^{rig}$ is locally Noetherian, we can without loss of generality take $T$ to be Noetherian. By Lemma \ref{lemma: triviality algebraic groups spaces on points} below, it suffices to show that the fibers of $\text{Aut}(p^{rig})$ over any geometric point of $T$ are trivial, and therefore we can assume without loss of generality that $T= Spec(k)$ for an algebraically closed field $k$. Then $p^{rig}$ is a $k$-point coming from a point $p \in \mathcal{M}{\rm{Hodge}}_{C_B}^{ss}$. The group automorphisms of $p^{rig}$ is just the quotient of the group automorphisms $\text{Aut}(p)$ by the central $\mathbb{G}_m$. Therefore it suffices to show that the group scheme of automorphisms of any point $p \in \mathcal{M}{\rm{Hodge}}_{C_B}^{ss}$ is equal to the constant scalars $\mathbb{G}_m$.

The automorphisms of a pair $x = (\mathcal{F}, \nabla)\in \mathcal{M}{\rm{Hodge}}_{C_B}^{ss}(k)$ consists of the automorphisms of the vector bundle $\mathcal{F}$ that commute with the logarithmic $t$-connection $\nabla: \mathcal{F} \to \mathcal{F} \otimes \omega_C(D)$. We have a closed immersion of algebraic groups $\mathbb{G}_m \subset \text{Aut}(x)$. Since $x$ is stable (which is the same as semistable because $n$ and $d$ are coprime), the usual argument (cf. \cite[pg. 90]{Simpson-repnI}) shows that $\mathbb{G}_m \hookrightarrow \text{Aut}(x)$ induces a bijection at the level of $k$-points, and so $\mathbb{G}_m$ must be the reduced subgroup scheme of $\text{Aut}(x)$. To show equality of schemes, it suffices to show that the scheme of automorphisms $\text{Aut}(x)$ is smooth over $k$, which would follow if we can prove that the Lie algebra of the group scheme of automorphisms is one-dimensional. But standard deformation theory shows that the Lie algebra consists of endomorphisms of $\mathcal{F}$ that commute with $\nabla$. By the same argument this just consists of the one-dimensional space of constant endomorphisms, as desired. 
\end{proof}

\begin{lemma} \label{lemma: triviality algebraic groups spaces on points}
Let $T$ be a Noetherian scheme, and let $G$ be an group algebraic space of finite type over $T$. Suppose that for all geometric points $\overline{t} \in T$, the fiber $G_{\overline{t}}$ is the trivial group scheme over $\overline{t}$. Then $G$ is the trivial group scheme over $T$.
\end{lemma}
\begin{proof}
Let $e: T \to G$ denote the identity section. We know that for all geometric points $\overline{t} \in T$, the restriction $e_{\overline{t}}: \overline{t} \to G_{\overline{t}}$ is an isomorphism. Since the property of being an isomorphism can be checked flat locally, this actually implies that for all points $t \in T$ we have that $e_t$ is an isomorphism. We want to conclude that $e$ is an isomorphism. Since this statement is \'etale local on $G$, after choosing an \'etale atlas $X \to G$, we just need to show that the monomorphism $e_{X}: T \times_{G} X \to X$ is an isomorphism. Consider the following commutative diagram of schemes.
\begin{figure}[H]
\centering
\begin{tikzcd}
  T \times_{G} X \ar[r, "e_X"] \ar[dr] & X \ar[d] \\  & T
\end{tikzcd}
\end{figure}
By assumption, for every point $t \in T$, the restriction $(e_{X})_t: t \times_{e} X_t \to X_t$ is an isomorphism. Note that $T \times_{e} X \to T$ is \'etale, and hence flat. By the fiberwise criterion for flatness \cite[\href{https://stacks.math.columbia.edu/tag/05VK}{Tag 05VK}]{stacks-project}, we conclude that $e_{X}$ is flat. So $e_{X}$ is a flat monomorphism of finite type, and hence an open immersion. We also know that $e_{X}$ is surjective, because it is an isomorphism over every point of $T$. Therefore $e_{X}$ is an isomorphism, as desired.
\end{proof}
\begin{lemma} \label{lemma:gm gerbe}
 Assume that the rank $n$ and degree $d$ are coprime.
There is an isomorphism $M{\rm{Hodge}}_{C_B}^{ss} \cong (\mathcal{M}{\rm{Hodge}}^{ss}_{C_B})^{rig}$. In particular the morphism $\mathcal{M}{\rm{Hodge}}_{C_B}^{ss} \to M{\rm{Hodge}}_{C_B}^{ss}$ is a smooth good moduli space morphism. 
\end{lemma}
\begin{proof}
By the universal property of good moduli spaces, we have a canonical morphism $\psi: (\mathcal{M}{\rm{Hodge}}^{ss}_{C_B})^{rig} \to M{\rm{Hodge}}_{C_B}^{ss}$ (in this case we could have also used the universal property of rigidifications). By applying \cite[Prop. 4.5]{alper-good-moduli} to the good moduli space morphism $f:\mathcal{M}{\rm{Hodge}}_{C_B}^{ss} \to (\mathcal{M}{\rm{Hodge}}^{ss}_{C_B})^{rig}$, we see that we have a natural isomorphism of functors $\psi_*(-) \simeq \psi_* f_* f^*(-)$. Since  $\psi \circ f$ is adequately affine (cf. \S\ref{subsection:Hodge with poles})
and  $f^*$ is exact by the smoothness of $f,$ we see that $\psi$ is adequately affine as well. Since $(\mathcal{M}{\rm{Hodge}}_{C_B}^{ss})^{rig}$ is an algebraic space, the morphism $\psi$ is affine \cite[Thm. 4.3.1]{alper_adequate}. Hence, it is an isomorphism because $\psi_*(\cO_{(\mathcal{M}{\rm{Hodge}}^{ss}_{C_B})^{rig}}) = \psi_* f_*(\cO_{\cM{\rm{Hodge}}^{ss}_{C_B}}) = \cO_{M{\rm{Hodge}}^{ss}_{C_B}}$ by the pushforward property in the definition of adequate moduli space morphism \cite[Defn. 5.1.1 (2)]{alper_adequate}. It follows that $\mathcal{M}{\rm{Hodge}}_{C_B}^{ss} \to (\mathcal{M}{\rm{Hodge}}^{ss}_{C_B})^{rig} \cong M{\rm{Hodge}}_{C_B}^{ss}$ is a smooth good moduli space morphism.
\end{proof}
\begin{remark} \label{remark: universal correpresentability}
Since good moduli space morphisms commute with arbitrary base-change, it follows from Lemma \ref{lemma:gm gerbe} that the formation of the good moduli space $M{\rm{Hodge}}_{C_B}^{ss}$ commutes with arbitrary base change over $\mathbb{A}^1_{B}$. This also follows from \cite[Thm. 1.1]{langer-moduli-lie-algebroids}. Note that, a priori, the formation of the GIT quotient in arbitrary characteristic is only known to commute with flat base change.
\end{remark}

\begin{coroll} \label{coroll: moduli space vs stack}
If the rank $n$ and degree $d$ are coprime, the moduli space $M{\rm{Hodge}}_{C_B}^{ss}$ is smooth over $\mathbb{A}^1_{B}$ if and only if the stack $\mathcal{M}{\rm{Hodge}}_{C_B}^{ss}$ is smooth over $\mathbb{A}^1_B$.
\end{coroll}
\begin{proof}
This is immediate from Lemma \ref{lemma:gm gerbe} and the fact that property of being a smooth morphisms can be checked smooth locally.
\end{proof}

\end{subsection}

\begin{subsection}{Proof of the smoothness and dimension assertions in Theorem \ref{thm:smooth}}\label{subs:pfsmooth}

In view of Corollary \ref{coroll: moduli space vs stack}, in order to prove the smoothness assertion in Theorem \ref{thm:smooth}, it suffices to show that the stack $\mathcal{M}{\rm{Hodge}}_{C_B}^{ss}$ is smooth over $\mathbb{A}^1_{B}$. Our proof of smoothness of the stack $\mathcal{M}{\rm{Hodge}}_{C_B}^{ss}$ applies even when the rank and degree are not coprime.

\begin{prop} \label{prop: smoothness of the stack}
Without coprimeness assumptions on the rank $n$ and degree $d$, the morphism of stacks $\mathcal{M}{\rm{Hodge}}_{C_B}^{ss} \to \mathbb{A}^1_B$ is smooth.
\end{prop}
\begin{proof}
We use the lifting criterion for smoothness \cite[\href{https://stacks.math.columbia.edu/tag/0DP0}{Tag 0DP0}]{stacks-project} for the finite type morphism $\mathcal{M}{\rm{Hodge}}_{C_B}^{ss} \to \mathbb{A}^1_B$. Since both the target and the source are locally Noetherian, by \cite[\href{https://stacks.math.columbia.edu/tag/02HT}{Tag 02HT}]{stacks-project} it suffices to show the existence of lifting for square-zero thickenings of local Artin algebras. 

Let $A$ be a local Artin $k$-algebra with maximal ideal $\mathfrak{m}$ and residue field $k$. Fix a morphism $Spec(A) \to \mathcal{M}{\rm{Hodge}}_{C_B}^{ss}$, inducing a composition $Spec(A) \to \mathbb{A}^1_B$, defining a function $t_A \in A.$
Choose a square-zero thickening $\widetilde{A} \to A$ with defining ideal $\widetilde{I}$ (cf.
\S\ref{sec:obs-cl}). Let $Spec(\widetilde{A}) \to \mathbb{A}^1_B$ 
be a choice of a morphism so that $Spec(A) \hookrightarrow Spec(\widetilde{A})$ is a morphism over $\mathbb{A}^1_B$. 
We need to find a lifting as in the dotted arrow below.
	\begin{figure}[H]
\centering
\begin{tikzcd}
  Spec(A) \ar[r] \ar[d, symbol = \hookrightarrow] &  \mathcal{M}{\rm{Hodge}}_{C_B}^{ss} \ar[d] \\ Spec(\widetilde{A}) \ar[r] \ar[ur, dashrightarrow] & \mathbb{A}^1_B.
\end{tikzcd}
\end{figure}
The family $Spec(A) \to \mathcal{M}{\rm{Hodge}}_{C_B}^{ss}$ is represented by a pair $x_A = (\mathcal{F}, \nabla)$, i.e. a logarithmic $t_A$-connection on $C_A$. By Proposition \ref{prop: obstruction class}, in order to show the existence of a lift for this family it suffices to prove that $\mathcal{Q}_{x_A} =0$.

By Nakayama's lemma,
if we show $\mathcal{Q}_{x_A} \otimes_{A} A/\mathfrak{m} = 0,$ then we have
$\mathcal{Q}_{x_A} =0.$ Using the compatibility of the obstruction module with base-change (Corollary \ref{coroll: base change obstruction module}), we see that $\mathcal{Q}_{x_A} \otimes_{A} A/\mathfrak{m} \cong \mathcal{Q}_{x_{A/\mathfrak{m}}}$, where $x_{A/\mathfrak{m}}$ is obtained by pulling-back $x_A$ to $C_{A/\mathfrak{m}}$. Therefore, without loss of generality, we can assume that $A = k$. In particular $x_A = x_k$ is a $k$-point of the stack $\mathcal{M}{\rm{Hodge}}_{C_B}^{ss},$ and $\mathcal{Q}_{x_A}=\mathcal{Q}_{x_k}$ is a
$k$-vector space. We also assume, without loss of generality, that $B = Spec(k)$, and that $k$ is algebraically closed.

Recall that there is a lift of the $\mathbb{G}_{m}$-action on $\mathbb{A}^1_k$ to the stack $\mathcal{M}{\rm{Hodge}}_{C_k}^{ss}$, given by scaling the universal logarithmic $t$-connection. Starting with our point $x_k$, we consider the morphism $\mathbb{G}_m \to \mathcal{M}{\rm{Hodge}}_{C_k}^{ss}$ induced by the action $y \mapsto y \cdot x_k$. 
In this family, the vector bundle $\mathcal{F}$ remains constant, and we scale the $t$-connection $\nabla$. 
This can be completed to a $\mathbb{G}_m$-equivariant morphism $\mathbb{A}^1_k \to \mathcal{M}{\rm{Hodge}}_{C_k}$ to the stack $\mathcal{M}{\rm{Hodge}}_{C_k}$ of all $t$-connections, with no semistability condition. The image of $0 \in \mathbb{A}^1_k$ is given by the pair $(\mathcal{F}, 0)$ consisting of the vector bundle $\mathcal{F}$ and the zero Higgs-field. 
Using the argument for the ``semistable reduction" theorem in \cite[Thm. 5.1]{langer2014semistable}, we modify $\mathbb{A}^1_k \to \mathcal{M}{\rm{Hodge}}_{C_k}$ to a $\mathbb{G}_m$-equivariant semistable family $x_{\mathbb{A}^1_k}: \mathbb{A}^1_k \to \mathcal{M}{\rm{Hodge}}_{C_k}^{ss}$.
Consider the corresponding $\mathbb{G}_m$-equivariant $k[z]$-module of finite type $N:=\mathcal{Q}_{x_{\mathbb{A}^1_k}}$, where $z$ is the coordinate of $\mathbb{A}^1_k$. By construction 
the fiber over $1\in \mathbb{A}^1_k$ is the $k[z]/(z-1)$-module $N_1=\mathcal{Q}_{x_k}$ that we are interested in. To show $N_1=0$, it suffices
to show that we have $N_0=0$ for the fiber at $0$.
If $x_0$ denotes the image of $0$ under $x_{\mathbb{A}^1_k}$, then $N_0 = \mathcal{Q}_{x_0}$. In this case,
$x_0$ lies on the $0$-fiber of the stack $\mathcal{M}{\rm{Hodge}}_{C_k}^{ss}$, and so it is represented by a logarithmic Higgs bundle $(\mathcal{F}_0, \nabla_0).$ 
Hence, the fact that $N_0=\mathcal{Q}_{x_0}=0$ follows from the computation for Higgs bundles given in Lemma \ref{lemma: obstruction module higgs bundles} below.

% in other words, we need to show that the support of $\mathcal{Q}_{x_{\mathbb{A}^1_k}}$ is empty.} 
% Since the support is closed and $\mathbb{G}_m$-stable, we can do this by proving that the fiber of $\mathcal{Q}_{\mathbb{A}^1_k}$ at the attracting point $0$ vanishes. 
% Therefore, we are reduced to looking at $\mathcal{Q}_{x_0}$, where $x_0$ is the image of $0$ under $x_{\mathbb{A}^1_k}$. In this case,
% $x_0$ lies on the $0$-fiber of the stack $\mathcal{M}{\rm{Hodge}}_{C_k}^{ss}$, and so it is represented by a logarithmic Higgs bundle $(\mathcal{F}_0, \nabla_0)$. 
% Hence, the fact that $\mathcal{Q}_{x_0}=0$ follows from the computation for Higgs bundles given in Lemma \ref{lemma: obstruction module higgs bundles} below.
\end{proof}

\begin{lemma} \label{lemma: obstruction module higgs bundles}
Suppose $B = Spec(k)$ for an algebraically closed field $k$. Let $x_0$ be a $k$-point of $\mathcal{M}{\rm{Hodge}}_{C_k}^{ss}$ in the $0$-fiber over $\mathbb{A}^1_k$, represented by a semistable logarithmic Higgs bundle $x_0 = (\mathcal{F}, \nabla)$. Then $\mathcal{Q}_{x_0} = 0$.
\end{lemma}
\begin{proof}
By a degeneration and semicontinuity argument, it suffices to prove $\mathcal{Q}_{x_0} =0$ for closed points of the stack. In other words, we can assume that $x_0$ represents a polystable logarithmic Higgs bundle. We need to prove the vanishing of  the cokernel $\mathcal{Q}_{x_0}$ of the  morphism (cf. Def. \ref{defn: obstruction module}):
\[H^1(\mathcal{E}nd(\mathcal{F})) \xrightarrow{H^1(\varphi_{x_0})} H^1( \mathcal{E}nd(\mathcal{F}) \otimes_{\mathcal{O}_{C}} \omega_{C/k}(D_k)).
\]
Under our assumptions, the commutator $\varphi_{x_0}$ is $\mathcal{O}_C$-linear. Therefore we can consider the dual 
twisted morphism  $\varphi_{x_0}^{\vee} \otimes_{\mathcal{O}_{C}} 
\text{id}_{\omega_{C/k}} : \mathcal{E}nd(\mathcal{F})^{\vee}(-D) \to \mathcal{E}nd(\mathcal{F})^{\vee} 
\otimes_{\mathcal{O}_{C}} \omega_{C/k}$. Under the identifications 
provided by Serre duality, the morphism $H^1(\varphi_{x_0})$ is identified with the dual
\[ H^0(\varphi_{x_0}^{\vee} \otimes_{\mathcal{O}_{C}} 
\text{id}_{\omega_{C/k}})^{\vee}: H^0(\mathcal{E}nd(\mathcal{F})^{\vee} \otimes_{\mathcal{O}_{C}} \omega_{C/k})^{\vee} 
\to H^0(\mathcal{E}nd(\mathcal{F})^{\vee}(-D))^{\vee}.\]
Therefore, the $k$-vector space $\mathcal{Q}_{x_0}$ is canonically isomorphic to the dual $\mathcal{K}^{\vee}$ of the following kernel:
\[ \mathcal{K} := \text{ker}\left[ H^0(\mathcal{E}nd(\mathcal{F})^{\vee}(-D))
\xrightarrow{H^0(\varphi_{x_0}^{\vee} 
\otimes_{\mathcal{O}_{C}} \text{id}_{\omega_{C/k}})} 
H^0(\mathcal{E}nd(\mathcal{F})^{\vee} \otimes_{\mathcal{O}_{C}} 
\omega_{C/k}) \right]. \]
We want to show that $\mathcal{K}$ vanishes. Note that there is a transposition isomorphism $\tau: \mathcal{E}nd(\mathcal{F}) \to \mathcal{E}nd(\mathcal{F})^{\vee}$ given by the swap
(transposition of matrices):
\[ \tau: \mathcal{E}nd(\mathcal{F}) = \mathcal{F}^{\vee}\otimes 
\mathcal{F} \xrightarrow{swap} \mathcal{F} 
\otimes \mathcal{F}^{\vee} \to \mathcal{E}nd(\mathcal{F})^{\vee}. \]
This also induces identifications $\mathcal{E}nd(\mathcal{F})^{\vee} \otimes \omega_{C/k} \cong \mathcal{E}nd(\mathcal{F})\otimes \omega_{C/k}$ and $\mathcal{E}nd(\mathcal{F})^{\vee}(-D) \cong \mathcal{E}nd(\mathcal{F})(-D)$. Consider the diagram of $\mathcal{O}_C$-modules:
	\begin{figure}[H]
\centering
\begin{tikzcd}
  \mathcal{E}nd(\mathcal{F})^{\vee}(-D) \; \; \; \;\ar[r, "\varphi_{x_0}^{\vee} \otimes_{\mathcal{O}_C} \text{id}_{\omega_{C/k}}"] \ar[d, "\tau \otimes_{\mathcal{O}_{C}} \text{id}_{\mathcal{O}_{C}(-D)}", labels = left] & \; \; \; \; \mathcal{E}nd(\mathcal{F})^{\vee} \otimes_{\mathcal{O}_{C}} \omega_{C/k} \ar[d, "\tau \otimes_{\mathcal{O}_{C}} \text{id}_{\omega_{C/k}}"] \\  \mathcal{E}nd(\mathcal{F})(-D) \; \; \; \ar[r,"\varphi_{x_0} \otimes_{\mathcal{O}_{C}} \text{id}_{\mathcal{O}_{C}(-D)}", labels = below]  & \; \; \;\mathcal{E}nd(\mathcal{F}) \otimes_{\mathcal{O}_{C}} \omega_{C/k}.
\end{tikzcd}
\end{figure}

 The diagram is commutative by the linear algebra fact that the dual of the commutator morphism of matrices is identified with the commutator morphism itself under transposition.
 
%  We claim that this diagram commutes; in other words we claim that the dual twisted morphism $\varphi_{x_0}^{\vee} \otimes_{\mathcal{O}_{C}} \Omega^1_{C/k}$ becomes $\varphi_{x_0}\otimes_{\mathcal{O}_{C}} \mathcal{O}_{C}(-D)$ under the transposition identifications. In order to show that the diagram above commutes, we can work locally on $C$ and pass to an affine open where all of $\mathcal{F}$, $\mathcal{O}_{C}(-D)$ and $\Omega^1_{C/k}$ are trivializable. After choosing trivializations, the commutativity of the above diagram amounts to the fact that the dual of commutator of matrices is identified with the commutator itself under transposition. This last fact is just a linear algebra computation, using the fact that taking the transponse of matrices is an antihomomorphism (reverses the order of multiplication).

From the commutativity of the diagram, we see that $\mathcal{K}$ is identified with the following kernel $\mathcal{K}_{D}:$
\[ \mathcal{K}_D := \text{ker}
\left[ 
H^0(\mathcal{E}nd(\mathcal{F})(-D) )
\xrightarrow{H^0(\varphi_{x_0}\otimes_{\mathcal{O}_{C}}
\text{id}_{\mathcal{O}_C(-D)})} H^0(\mathcal{E}nd(\mathcal{F}) \otimes_{\mathcal{O}_{C}} \omega_{C/k}) \right].\]
The inclusion $\mathcal{O}_{C}(-D) \hookrightarrow \mathcal{O}_{C}$ induces an inclusion of vector spaces $\mathcal{K}_D \subset \mathcal{G}$, where $\mathcal{G}:= \text{ker}(H^0(\varphi_{x_0}))$ is the subset global of endomorphisms of $\mathcal{F}$ that commute with the Higgs field $\nabla$. Since $k$ is algebraically closed and $(\mathcal{F}, \nabla)$ is polystable, we know that $\mathcal{G}$ consists of a direct sum of ``constant" matrix endomorphisms in $M_{n_i \times n_i}(k)$ of $\mathcal{F}$ that act on each isotypic component of $\mathcal{F}$ consisting of a direct sum of $n_i$ isomorphic stable logarithmic Higgs bundles (cf. the proof of Lemma \ref{lemma:gm gerbe}). Notice that $\mathcal{K}_D \subset \mathcal{G}$ is the subset of endomorphisms in $\mathcal{G}$ that vanish on the divisor $D$. But any nonzero ``constant" matrix in $\mathcal{G} \setminus \{0\}$ is nowhere vanishing. Since $D$ is nonempty, we conclude
that $\mathcal{K}_D =0$, as desired.
\end{proof}

\begin{remark}\label{stsmoospno}
If we have $2g-2+ \text{deg}(D) \geq 2$, then the strictly semistable points of the moduli space of logarithmic Higgs bundles are singular points (the same holds for the moduli space of logarithmic connections). Therefore, under the assumption $2g-2+ \text{deg}(D) \geq 2$, the moduli space is singular in the non-coprime case, even though we know that the stack of semistable objects is smooth by Lemma \ref{prop: smoothness of the stack}.
\end{remark}

\begin{coroll}
Suppose that $n$ and $d$ are coprime. For any point $a \in \mathbb{A}^1_B$, the fiber $(M{\rm{Hodge}}_{C_B}^{ss})_a$ of the moduli space is equidimensional of dimension $n^2(2g-2+ {\rm \text{deg}}(D_a))+1$.
\end{coroll}
\begin{proof}
In view of the smoothness of $(M{\rm{Hodge}}_{C_B}^{ss})_a$, it suffices to prove that for every closed geometric point $x \in (M{\rm{Hodge}}_{C_B}^{ss})_a$ the dimension of the tangent space $T_{(M{\rm{Hodge}}_{C_B}^{ss})_a, x}$ of $x$ in $(M{\rm{Hodge}}_{C_B}^{ss})_a$ is equal to $n^2(2g-2+ {\rm deg}(D_a))+1$. Choose a lift $\widetilde{x}$ of $x$ in the stack $(\mathcal{M}{\rm{Hodge}}_{C_B}^{ss})_a$. By Lemma \ref{lemma:gm gerbe}, it follows that the morphism $(\mathcal{M}{\rm{Hodge}}_{C_B}^{ss})_a \to (M{\rm{Hodge}}_{C_B}^{ss})_a$ is a $\mathbb{G}_m$-gerbe. In other words, \'etale locally on $(M{\rm{Hodge}}_{C_B}^{ss})_a$  the fibers of $(\mathcal{M}{\rm{Hodge}}_{C_B}^{ss})_a \to (M{\rm{Hodge}}_{C_B}^{ss})_a$ are isomorphic to the classifying stack $B\mathbb{G}_m$. This implies, by the definition of tangent space, that $(\mathcal{M}{\rm{Hodge}}_{C_B}^{ss})_a \to (M{\rm{Hodge}}_{C_B}^{ss})_a$ induces an isomorphism of tangent spaces $T_{(\mathcal{M}{\rm{Hodge}}_{C_B}^{ss})_a, \widetilde{x}} \xrightarrow{\sim} T_{(M{\rm{Hodge}}_{C_B}^{ss})_a, x}$. By Corollary \ref{coroll: relative tangent space}, we have
\[ \text{dim}\left(T_{(M{\rm{Hodge}}_{C_B}^{ss})_a, x}\right) = \text{dim}\left(T_{(\mathcal{M}{\rm{Hodge}}_{C_B}^{ss})_a, \widetilde{x}}\right) = n^2(2g-2+ {\rm \text{deg}}(D_a))+1 + \text{dim}(\mathcal{Q}_{\widetilde{x}}) \]
By the vanishing of the obstruction module $\mathcal{Q}_{\widetilde{x}}$ proven in Proposition \ref{prop: smoothness of the stack}, it follows that $\text{dim}\left(T_{(M{\rm{Hodge}}_{C_B}^{ss})_a, x}\right) = n^2(2g-2+ {\rm \text{deg}}(D_a))+1$, as desired.
\end{proof}
\end{subsection}

\begin{subsection}{Smoothness in the coprime case without poles}\label{smooth:nopoles}
In this subsection, we consider flat connections without poles.
By \cite[Prop. 3.1]{biswas-subramanian-weil-criterion}, a vector bundle on a curve over an algebraically closed field 
admits a flat connection if and only if each of its indecomposable summands
has degree not invertible in the field.

In particular, in characteristic zero, the de Rham moduli stack is empty
unless the degree is zero. In degree zero, the de Rham moduli space of semistable flat connections is singular, and similarly for the Higgs and Hodge moduli spaces.

In positive characteristic $p,$ in the case when the degree $d=d'p$ is a multiple
of $p$  and rank and degree are coprime $(n,d'p)=1$, the smoothness has been
proven  in \cite[Prop. 3.1]{decataldo-zhang-nahpostive} under the assumption that the base $B$ is reduced and Noetherian. The proof  given in loc. cit. is ad hoc and based on earlier related smoothness results.

The methods in the proof of Theorem \ref{subs:pfsmooth} can be modified to prove the smoothness of the Hodge moduli space $M{\rm{Hodge}}^{ss}_{C_B} \to \mathbb{A}^1_{B}$ when: $D$ is the empty divisor, $n$ and $d$ are coprime, and the integer $d$ maps to zero in all residue fields of points of the not necessarily reduced but Noetherian $B.$
Note that when $B$ is connected, these conditions can be met only 
when $B$ has positive characteristic, say, $p,$ so that we are then in the aforementioned case where $(n,d=d'p)=1$ with $B$ Noetherian.

Here, we give a sketch of the proof of the smoothness assertion made above. For this remark we need to keep track of the rank $n,$ and so we use the notation $M{\rm{Hodge}}^{ss}_{n,C_B}$ for the moduli space.

We need to prove that the obstruction $\text{ob}_{x}$ vanishes for any given geometric point $x: Spec(k) \to \cM{\rm{Hodge}}_{n,C_B}$. We can assume that $B = Spec(k)$. If we write $x=(\cF, \nabla)$, then the determinant connection $\text{det}(x):= (\text{det}(\cF), \text{det}(\nabla))$ is an element in $\cM {\rm{Hodge}}_{1,C_k}$. There is a commutative diagram of morphisms induced by the trace $\text{tr}: \cE{nd}(\cF) \to \cO_{C}:$ 
 	\begin{figure}[H]
\centering
\begin{tikzcd}
  H^1(\mathcal{E}nd(\mathcal{F})) \ar[r, "H^1(\varphi_{x})"] \ar[d, "H^1(\text{tr})"] & \; \; \; \; H^1(\mathcal{E}nd(\mathcal{F})\otimes_{\mathcal{O}_{C}} \omega_{C/k}) \ar[d, "H^1(\text{tr}\otimes \text{id}_{\omega_{C/k}})"] \\  H^1(\cO_{C}) \; \; \; \ar[r,"H^1(\varphi_{\text{det}(x)})"]  & \; \; \;H^1( \omega_{C/k}),
\end{tikzcd}
\end{figure}
\noindent which induces a trace map $\mathcal{Q}_{x} \to \mathcal{Q}_{\text{det}(x)}$ on the
cokernels. It can be checked directly from the construction that this maps 
sends $\text{ob}_{x}$ to $\text{ob}_{\text{det}(x)}$. Since $d$ is 
assumed to be divisible by the characteristic of $k,$ it follows that every line 
bundle of degree $d$ admits a $t$-connection. The Hodge stack 
$\cM {\rm{Hodge}}_{1,C_k}$ is isomorphic to a smooth affine bundle 
with fibers $H^0(C, \omega_{k})$ over the Picard stack. In particular 
$\cM {\rm{Hodge}}_{1,C_k}$ is smooth, and so $\text{ob}_{\text{det}(x)} = 0$. 
This shows that $\text{ob}_{x}$ lies in the kernel of the trace 
morphism $\mathcal{Q}_{x} \to \mathcal{Q}_{\text{det}(x)}$, and 
so it lies in the trace zero part  of this module. 
One verifies that this latter is the 
trace zero obstruction module $\mathcal{Q}^0_{x}$ 
formed using trace-zero endomorphisms $\mathcal{E}nd^0(\cF).$
Hence, it suffices to show that the trace zero obstruction module $\mathcal{Q}^0_{x}$ vanishes. 
The same degeneration argument as in the proof of Theorem \ref{thm:smooth} 
above shows that we can take $x$ to be a Higgs bundle. Since the divisor $D$ 
is empty, the obstruction module $\mathcal{Q}_{x}$ is dual to the space 
$\mathcal{K}$ of endomorphism of the Higgs bundle. This consists of the 
constant scalar matrices $k$, because $n$ and $d$ are coprime. The trace-zero 
obstruction module $\mathcal{Q}_{x}^0$ will be dual to the space of trace-zero 
endomorphisms. In other words, the trace-zero module $\mathcal{Q}_{x}^0$ is 
isomorphic to the dual of the kernel of the trace on scalar matrices 
$k \xrightarrow{n \cdot(-)} k$. Since $n$ is coprime to $d=d'p$, it is also 
coprime to $\text{char}(k)$, and hence this kernel is $0.$

\end{subsection}

\begin{subsection}{Proof of the integrality assertion in Theorem \ref{thm:smooth}}
\begin{prop} \label{prop: irreducibility of fibers of Hodge}
All of the fibers of $M{\rm{Hodge}}_{C_B}^{ss} \to \mathbb{A}^1_B$ 
are smooth and geometrically integral.
\end{prop}

\begin{context}
In order to show the proposition, we can assume without loss of generality that $B = Spec(k)$ is a field. We shall assume this for the rest of this section.
\end{context} 

We start by proving the proposition for the $0$-fiber 
$M{\rm Higgs}_{C_k}^{ss}$. 
This is the moduli space of logarithmic Higgs bundles with poles at $D$.
We use the Hitchin fibration $M{\rm{Higgs}}_{C_k}^{ss} \to A(C)$, where $A(C)$ 
denotes the Hichin base $A(C) = \bigoplus_{i=1}^n 
H^0((\omega_{C/k}(D))^{\otimes i})$, viewed as an affine space over $k$.

\begin{defn}[Spectral curve]
Let $W = {\rm{Spec}}({\rm{Sym}}(\omega_{C/k}(D)^{\vee}))$  be the total space
of the line bundle $\omega_{C/k}(D)$, with projection 
$\pi_{W}: W \to C$. There is the tautological section 
$x: \mathcal{O}_{W} \to \pi_W^*(\omega_{C/k}(D))$. 
For any morphism $Spec(k) \to A(C)$ corresponding to a tuple 
of sections $\left(\sigma_i 
\in H^0(\omega_{C/k}(D))^{\otimes i})\right)_{i=1}^n$, 
we define the spectral curve $C_{(\sigma_i)} \subset W$ to 
be the vanishing locus of the section:
\[ x^n + \pi_{W}^*(\sigma_1) x^{n-1} + \ldots + \pi_W^*(\sigma_{n-1}) x + \pi_W^*(\sigma_n)  \, \in \, H^0(\pi_W^*(\omega_{C/k}(D))^{\otimes n}).
\]
\end{defn}

\begin{lemma} \label{lemma: generic spectral curve smooth}
Suppose that $k = \overline{k}$. 
\begin{enumerate}[(1)]
    \item The spectral curve assigned to the generic point of $A(C)$ is singular if and only if $g=0$, $n>1$ and $\text{deg}(D)=1$.
    \item The generic spectral curve is reducible if and only if $g=0$, $n>1$ and $\text{deg}(D)=2$.
\end{enumerate}
\end{lemma}
\begin{proof}
When $g=0$, $n>1$ and $\text{deg}(D)=1$, the Hitchin base consists of a single point corresponding to the $0$ section. The unique spectral curve is then an $n^{th}$ infinitesimal thickening of $C$, and therefore it is singular and irreducible.

We are left with the following remaining cases.
\begin{enumerate}[(A)]
    \item $n=1$.
    \item $(\omega_{C/k}(D))^{\otimes n}$ is very ample on $C$.
    \item $g=0$, $n>1$, and $\text{deg}(D) =2$.
    \item $g=1$, $n=2$, and $\text{deg}(D)=1$.
\end{enumerate}
Since smoothness is an open condition, it suffices to show that there exists a single spectral curve that is smooth to conclude smoothness. The same holds for (geometric) integrality.

The first case (A) is clear, because then every spectral curve is isomorphic to $C$. On the other hand, (B) follows from an application of Bertini's theorem (\cite[\href{https://stacks.math.columbia.edu/tag/0FD6}{Tag 0FD6}]{stacks-project}+ \cite[\href{https://stacks.math.columbia.edu/tag/0G4F}{Tag 0G4F}]{stacks-project}.

In case (C), we have that $H^0(\omega_{\mathbb{P}^1_k/k}(D))^{\otimes i}) = H^0(\cO_{\mathbb{P}^1_k}) = k$. For generic choice of constants $\sigma_i \in k$, the polynomial $x^n + \sigma_1x^{n-1} + \ldots \sigma_n$ in $k[x]$ splits into distinct linear factors, and then the corresponding spectral curve will be a disjoint union of $n$ copies of $C$. Hence it will be smooth and reducible.

 Let us assume (D) with $\text{char}(k) \neq 2$. For a section $\sigma_2 \in H^0( (\omega_{C/k}(D))^{\otimes 2})= H^0(\cO_{C}(2D))$, we consider the spectral curve $f: C_{\sigma_2} \to C$ defined by $x^n + \pi_W^*(\sigma_2)$ inside $W$. Let $c$ be a closed point of $C$. Choose a uniformizer $t$ for the completion of the local ring $\mathcal{O}_{C,p}$, and a trivialization of the stalk of $\omega_{C/k}(D)$ at $p$. Using these choices, we can write the formal fiber of $f: C_{\sigma_2} \to C$ at $p$ as $Spec(k\bseries{t}[x]/(x^2 + \sigma_2(t))$. By the Jacobi criterion for smoothness, $C_{\sigma_2}$ will be smooth at the points lying over $t$ if the following three polynomials in $k[x]$ don't have a common root
\[ \begin{cases}
 x^2 +\sigma_n(0)\\
 2x \\
 \partial_t(\sigma_n)(0)
\end{cases}\]
Since $2$ is coprime to the characteristic of $k$, the second equation forces $x$ to be $0$. Therefore the points lying over $p$ will be smooth if one of $\sigma_2(0)$ or $\partial_t(\sigma_2)(0)$ does not vanish. This is true as long we choose a section $\sigma_2 \in H^0(\cO_{C}(2D))$ whose vanishing locus consists of two distinct points of $C$, which is always possible. 

Moreover, for this choice of $\sigma_2$ the spectral curve $C_{\sigma_2}$ is integral, even if the characteristic is $2$. Indeed, the spectral curve $C_{\sigma_2} \to C$ is flat over $C$, since it is a relative global complete intersection over $C$ \cite[\href{https://stacks.math.columbia.edu/tag/00SW}{Tag 00SW}]{stacks-project}. Therefore it suffices to check integrality of the generic fiber. So we think of $\sigma_2$ as an element of the ring of functions $k(C)$, and we want to show that $\Spec(k[x]/(x^2 -\sigma_2))$ is integral. This is true because $\sigma_2$ is not a square in $k(C)$ ($\sigma_2$ has simple zeroes by construction). Therefore the generic spectral curve is integral in case (D) regardless of characteristic.

We are left to show the smoothness in case (D) with $\text{char}(k) =2$. Choose a nonzero $\sigma_1 \in H^0(\omega_{C/k}(D)) = H^0(\cO_{C}(D))$, and choose $\sigma_2 \in H^0(\omega_{C/k}(D)^{\otimes 2}) = H^0(\cO_{C}(2D))$ linearly independent to $(\sigma_1)^2$. Notice that $\sigma_1$ has only one zero at $D$, and it is a simple zero. On the other hand $\sigma_2$ does not vanish at $D$, since the linear system spanned by $(\sigma_1)^2$ and $\sigma_2$ is base-point free. Using that the characteristic is $2$, the local Jacobi criterion in this case tells us that the spectral curve is smooth at a point $0$ with uniformizer $t$ whenever the following polynomials in $k[x]$ don't have a common zero:
\[ \begin{cases}
 x^2 + x \sigma_1(0) + \sigma_2(0)\\
 \sigma_1(0) \\
 x\partial_t(\sigma_1)(0) + \partial_t(\sigma_2)(0)
\end{cases}\]
Since $\sigma_1$ only vanishes at $D$, the second equation forces the point $0$ to be $D \in C(k)$. The morphism $C \to \mathbb{P}^1_k$ corresponding to the linear series spanned by $(\sigma_1)^2$ and $\sigma_2$ is ramified at $D$, and therefore we have $\partial_t(\sigma_2)(0) =0$. Since $D$ is a simple zero of $\sigma_1$, we have $\partial_t(\sigma_1)(0) \neq 0$, and so the vanishing of the third equation would imply $x=0$. Going back to the first equation, we see that the vanishing of all three equations forces $0$ to be the point $D$ and $\sigma_2(0)=0$, which is not true since $\sigma_2$ does not vanish at $D$.
\end{proof}

\begin{lemma}
$M{\rm{Higgs}}_{C_k}^{ss}$ is smooth and geometrically connected. It is empty when $n>1$, $g=0$ and $\text{deg}(D)\leq 2$.
\end{lemma}
\begin{proof}
We have already shown smoothness in Theorem \ref{thm:smooth}, we just need to prove geometric connectedness. For this we can replace $k = \overline{k}$. We start by dealing with the cases when the generic spectral curve is smooth and irreducible, as characterized in Lemma \ref{lemma: generic spectral curve smooth}. Note that $M{\rm{Higgs}}_{C_k}^{ss} \to A(C)$ is flat by miracle flatness \cite[\href{https://stacks.math.columbia.edu/tag/00R4}{Tag 00R4}]{stacks-project}, because it is a morphism between integral $k$-smooth schemes (Theorem \ref{thm:smooth}) with equidimensional fibers of the same dimension \cite[Cor. 8.2]{chaudouard-laumon}. By flatness, it suffices to show that the generic fiber of the Hitchin fibration is irreducible. The generic fiber of the Hitchin morphism will be a connected component of fixed degree of the Picard scheme associated to the smooth and irreducible generic spectral curve \cite[Prop. 3.6]{bnr-spectral-curves}, \cite[\S5]{schaub-courbes-spectrales}. Therefore it is connected, as desired.

We are left with the case when $n>1$, $g=0$, and $\text{deg}(D)\leq 2$. Then there are no stable Higgs bundles under our coprime assumption $(n,d)=1$. Therefore the Higgs moduli space is empty in that case, and hence vacuously irreducible.
\end{proof}

Now we are ready for the proof of the more general proposition.
\begin{proof}[Proof of Proposition \ref{prop: irreducibility of fibers of Hodge}]
We have already seen in Theorem \ref{thm:smooth} that all fibers are smooth; we only need to prove that they are geometrically connected. Without loss of generality, we replace $B$ with $Spec(k)$ for $k =  \overline{k}$. We already know that the $0$-fiber $M{\rm{Higgs}}_{C_k}^{ss}$ is geometrically connected. Using this and the ``semistable reduction" theorem in \cite[Thm. 5.1]{langer2014semistable}), we see that the total space $M{\rm{Hodge}}_{C_k}^{ss}$ is connected. Since $M{\rm{Hodge}}_{C_k}^{ss}$ is smooth, this means that $M{\rm{Hodge}}_{C_k}^{ss}$ is integral. Therefore the generic fiber of $M{\rm{Hodge}}_{C_k}^{ss} \to \mathbb{A}^1_k$ is integral. Since $M{\rm{Hodge}}_{C_k}^{ss} \to \mathbb{A}^1_k$ is a constant family away from $0$, and $k = \overline{k}$, this means that all the fibers away from $0$ are also geometrically irreducible.
\end{proof}

\end{subsection}

\end{section}

\begin{section}{Proof of the cohomological Theorems \ref{thm:spk}, \ref{thm:spv}}\label{section:pfcohtms}

Theorems \ref{thm:spk}, \ref{thm:spv}  are concerned with specialization morphisms in the context of
moduli spaces of $t$-connections with poles and coprime rank and degree.

If the morphism to a DVR is not proper, as it is the case for  the moduli spaces above,
then the desirable specialization morphisms may fail to be defined. 
The paper \cite{decataldo-cambridge} studies this problem
and provides criteria for the existence of specialization morphisms. These criteria are often based
on the existence of a suitable completion of the morphism to the DVR, where one leverages
the existence of the specialization morphism after the completion to deduce the existence
before the completion.

In this section: we  recall some of the techniques, rooted in \cite{decataldo-cambridge} and \cite{decataldo-zhang-completion}, and employed in 
\cite{decataldo-zhang-nahpostive} to study specialization morphisms for moduli spaces of $t$-connections without poles
under  suitable coprimality assumptions; we  recall the constructions of the completions of moduli spaces
used in this study; we observe that these techniques apply to the case of poles;   we finally  prove Theorems \ref{thm:spk}, \ref{thm:spv}. 

\begin{subsection}{Completion of Hodge, Higgs and de Rham moduli spaces}\label{subs:1100}
In the remainder of this paper, 
we need  suitable completions of the structural morphisms
 $v_{{\rm{Hodge}},B}$ (\ref{notn:vHodgeB}),
$v_{{\rm{Higgs}},B}$(\ref{notn:vhiggsB}),
$v_{{\rm de \, Rham},B}$(\ref{notn:vdrB})  to the Noetherian base $B$
and of the structural morphism 
$\tau_{B}$ (\ref{notn:tau_B}) to $\mathbb{A}^1_B.$

\cite[\S2,4]{decataldo-zhang-completion} develops a general compactification technique and applies it to Hodge, Higgs and de Rham moduli spaces
without poles. This can also be applied
to the moduli spaces of $t$-connections with poles appearing in this paper, as soon as we have 
%their smoothness, which we do have by Theorem \ref{thm:smooth}, and 
the properness of the Hodge-Hitchin morphism,
which we do by \cite[Thm. 5.2]{langer-moduli-lie-algebroids}. These techniques give us all the desired completions, except for the morphism 
$v_{{\rm{Hodge}},B}$ (\ref{notn:vHodgeB}).
We complete the morphism $v_{{\rm{Hodge}},B}$ (\ref{notn:vHodgeB}) by
means of a simple additional construction, akin to the completion 
of  $\mathbb{A}^1_B$ given by $\mathbb{P}^1_B.$

Let us summarize the construction of all these completions,
and list the properties relevant to the proof of Theorems
\ref{thm:spk} and \ref{thm:spv}.

\begin{context}
For the compactification results in this section, we do not assume that the rank and degree are coprime, or that the fibers of the divisor 
$D$ of poles are nonempty.
\end{context}

\begin{notn}
In what follows,
we omit many decorations, and the moduli spaces in question may be with or without poles.
\end{notn} 

We have the following commutative diagram with Cartesian square of $\mathbb{G}_m$-equivariant morphisms
(see \cite[(48)]{decataldo-zhang-completion}):
\begin{equation}\label{act}
\xymatrix{
M:= M{\rm{Hodge}} \times_{\mathbb A^1} \mathbb A^2 \ar[r] \ar[d] 
\ar@/_2pc/[dd]_-{\tau'} & M{\rm{Hodge}} \ar[d]^-\tau  
& 
&
\\
\mathbb A^2_{x,y} \ar[r] \ar[d] & \mathbb A^1_t, 
& 
(x,y) \ar@{|->}[r] \ar@{|->}[d] 
& 
t=xy
\\
\mathbb A^1_x 
&  
&
x, 
&
}
\end{equation}
where the $\mathbb{G}_{m}$ action on $\mathbb A^2_{x,y}$ is defined by setting $\lambda(x,y):= (x,\lambda y)$, the 
$\mathbb{G}_m$-action  on $\mathbb A^1_\lambda$ is the usual dilation
$\lambda\cdot t =  \lambda t,$ and the $\mathbb{G}_m$  action on $\mathbb A^1_x$ is trivial.

The completions of $v_{{\rm{Hodge}},B}$ (\ref{notn:vHodgeB}),
$v_{{\rm{Higgs}},B}$(\ref{notn:vhiggsB})
and of the structural morphism 
$\tau_{B}$ (\ref{notn:tau_B}) to $\mathbb{A}^1_B$
are obtained as follows.
 We refer to \cite[Theorem 2.13 and (48) (resp. Theorem 2.14  and (49), if we wish to incorporate the Hitchin-type morphisms)]{decataldo-zhang-completion} for more details. Note that these theorems follow from the generalization \cite[Theorem 2.7]{decataldo-zhang-completion} of
a well-known  compactification technique of Simpson's, generalized in \cite[Theorem 2.7]{decataldo-zhang-completion}.

Recall that the nilpotent cone $N{\rm Higgs}$ is the fiber of the 
proper Hitchin morphism $h: M{\rm{Higgs}} \to A$ over the origin $o_A$ of the Hitchin base $A.$

\begin{defn}
We define $M^* \subset M$ as the open complement of the union of all nilpotent cones in the preimage $M_{x=0}$ of the $x$-axis inside $\mathbb{A}^2_{x,y}$.
\end{defn}

\begin{defn}
We define the following $\mathbb{A}^1_x$-schemes obtained by taking quotients by the $\mathbb{G}_m$-action:
\begin{itemize}
    \item $\overline{M{\rm{Hodge}}} := (M^*)/\mathbb{G}_m$ (proper over $\mathbb{A}^1_x,$ but not over $B$);
    \item $\overline{M{\rm{Higgs}}} := ((M^*)_{x=0})/\mathbb{G}_m$ (proper over $B$);
    \item $\overline{M\rm{de\,Rham}} := ((M^*)_{x=1})/\mathbb{G}_m$ (proper over $B$);
    \item $\partial\overline{M{\rm{Higgs}}} := ((M^*)_{x=0,y=0})/\mathbb{G}_m
``=" ((M^*)_{x=1,y=0})/\mathbb{G}_m =\partial \overline{M{\rm de \, Rham}};$
\end{itemize}
\end{defn}

Note that: $\partial \overline{M{\rm de \, Rham}} = \partial \overline{M{\rm{Higgs}}}  =
(M{\rm Higgs} \setminus N{\rm Higgs})/\mathbb{G}_m$
(proper over $B$).

The resulting proper morphism $\overline{\tau}: \overline{M{\rm{Hodge}}} \to \mathbb{A}^1_x$
is $\mathbb{G}_m$-equivariant for the natural $\mathbb{G}_m$-action on $\mathbb{A}^1_x$ given by 
$t\cdot x = tx$. 
After restriction over $\mathbb{G}_m \subseteq \mathbb{A}^1_x,$ 
we have $\mathbb{G}_m$-equivariant isomorphisms:
 \begin{equation}\label{a12p1}
 (\partial \overline{M{\rm{Hodge}}}, \overline{M{\rm{Hodge}}}, 
{M{\rm{Hodge}}})_{\mathbb{G}_m}
\simeq 
(\partial \overline{M{\rm de \, Rham}}, \overline{M{\rm de \, Rham}}, 
{M{\rm de \, Rham}}) \times \mathbb{G}_m.
\end{equation}

In particular, we have  natural isomorphisms:
$\partial \overline{M{\rm{Hodge}}} = (M^*)_{y=0}/\mathbb{G}_m \simeq 
\partial \overline{M{\rm{Higgs}}} \times \mathbb{A}^1 =
\partial \overline{M{\rm de \, Rham}} \times \mathbb{A}^1$
 (proper over $\mathbb{A}^1_B,$ but not over $B$).

\begin{notn} If in place of (\ref{act}) (i.e. \cite[(48)]{decataldo-zhang-completion}), we consider
its version \cite[(49)]{decataldo-zhang-completion}, augmented by the Hodge-Hitchin morphism, we obtain 
that the completions above factor through suitable completions of
the Hodge-Hitchin, Hitchin and de Rham-Hitchin morphisms, with suitably completed targets. We thus have the following three completions of morphisms  (cf. \cite[Theorems
2.19, 2.18 and 2.14, respectively]{decataldo-zhang-completion}):
\begin{itemize}
    \item A completion of the morphism $v_{{\rm{Higgs}},B}: M{\rm{Higgs}} \stackrel{h_{{\rm{Higgs}}}} \to A(C_B) \to  B$
to a morphism:
\begin{equation}\label{cohi}
\xymatrix{
\overline{v_{{\rm{Higgs}}}}: \overline{M{\rm{Higgs}}} 
\ar[rr]^-{{\overline{h_{{\rm{Higgs}}}}}}
&
&
 (\overline{A(C_B)})  \to  B.
}
\end{equation}
    \item (If $B$ has positive equicharacteristic) A completion of the morphism $v_{{\rm de \, Rham},B}: M{\rm de \, Rham} 
\stackrel{h_{{\rm de \, Rham}}} \to A(C_B)  \to  B$
to a morphism:
\begin{equation}\label{code}
\xymatrix{
\overline{v_{{\rm de \, Rham}}}: \overline{M{\rm de \, Rham}} 
\ar[rr]^-{{\overline{h_{{\rm de \, Rham}}}}}
&
&
 \overline{A(C_B)}  
\to  B.
}
\end{equation}
    \item (If $B$ has positive equicharacteristic) A completion of the morphism $\tau_{B}: M{\rm{Hodge}} 
\stackrel{h_{{\rm{Hodge}}}}\to A(C^{(B)}_B ) \times_B \mathbb{A}^1_B \to  \mathbb{A}^1_B$
to a morphism: 
\begin{equation}\label{cotau}
\xymatrix{
\overline{\tau_{B}}: \overline{M{\rm{Hodge}}} 
\ar[rr]^-{\overline{h_{{\rm{Hodge}}}}}
&
&
 \overline{A(C^{(B)}_B)}\times_B \mathbb{A}^1_B \to  \mathbb{A}^1_B.
}
\end{equation}
\end{itemize}
\end{notn}

The compositum morphism 
$\overline{v_{\rm Hodge}}: \overline{M{\rm{Hodge}}} \to  \mathbb{A}^1_B \to B$
is not proper as soon as the intermediate morphism to $\mathbb{A}^1_B$ 
is surjective (e.g.: in the case of coprime rank and degree
and nonempty divisor of poles $D$; or in the case of 
degree zero and empty $D$)  and therefore  does not yield the desired completion.

Next, we construct such a completion.

\begin{defn}
Let $\widetilde{\overline{M{\rm{Hodge}}}}$ denote the scheme over $\mathbb{P}^1_B$ obtained by gluing $\overline{M{\rm{Hodge}}}$ to $\overline{M{\rm de \, Rham}}\times \mathbb{A}^1_B$
along their open subsets over $\mathbb{G}_m$ by using the isomorphism (\ref{a12p1}) and using the same prescription
that yields $\mathbb{P}^1$ from two copies of $\mathbb{A}^1.$
In particular, we obtain 
a proper morphism
${\widetilde{\overline{\tau}}}:\widetilde{\overline{M{\rm{Hodge}}}} \to \mathbb{P}^1_B$.
\end{defn}

If $B$ is of positive equicharacteristic, then
we get a canonical factorization: 
\begin{equation}\label{eq:lok}
\xymatrix{
\widetilde{\overline{v_{{\rm{Hodge}}}}}:
\widetilde{\overline{M{\rm{Hodge}}}} 
\ar[rr]^-{\widetilde{\overline{h_{{\rm{Hodge}},B}}}}
\ar@/^2pc/[rrrr]^-{\widetilde{\overline{\tau}}}
&&
\overline{A(C^{(B)}_B}) \times_B \mathbb{P}^1_B
 \ar[rr]^-{\text{proj}}
 &&
 \mathbb{P}^1_B
 \ar[r]^-{\text{proj}}
 & B.
}
\end{equation}

The boundary $\partial \widetilde{\overline{M{\rm{Hodge}}}}= \partial' \cup \partial''$, complement of
$M{\rm{Hodge}},$ is made of two relative to $B$ hypersurfaces where $\partial''$ is the preimage of $\infty_B$
via the morphism to $\mathbb{P}^1_B$ and $\partial'$
is the closure of $\partial \overline{M{\rm{Hodge}}}.$

We have two charts $\overline{M{\rm{Hodge}}} = M^*/\mathbb{G}_m$ and  $\overline{M{\rm de \, Rham}}\times \mathbb{A}^1 = ((M^*)_{x=1}/ \mathbb{G}_m)  \times \mathbb{A}^1$, and in each of  these
two charts, before taking the quotient by $\mathbb{G}_m$,
the hypersurfaces $\partial'$ and $\partial''$ are given by relative Cartier divisors.

The key observation here is that, when the Hodge moduli space is smooth, e.g. in our case 
with poles and coprime rank and degree (cf. Theorem \ref{thm:smooth}), 
these Cartier divisors form a simple normal crossing divisor over $B$.
\end{subsection}
\begin{subsection}{Vanishing of vanishing cycles}
We return to our assumption that we are working with poles and that rank and degree are coprime. For this section we assume that $B$ is a DVR.
We denote by $\phi$ the vanishing cycle functors associated with morphisms to $B.$ We denote by $\overline{\mathbb{Q}}_{\ell, X}$ the 
$\overline{\mathbb{Q}}_{\ell}$-adic constant sheaf of rank one on a scheme $X$; we drop the space decoration if it is clear in the context.

Let us start by proving the following complement to SGA.

\begin{lemma}\label{cplmt sga}
Let $X$ be a Noetherian regular
scheme.
Let $v: X\to B$
be a smooth morphism from a Noetherian regular scheme to a DVR. Let $a: D\to X$ be a closed embedding with $D$ a divisor in simple normal crossings relative to $v;$ in particular, the irreducible components of $D$ are smooth over $B.$  
Assume that the prime $\ell$ is invertible in $X$ and in $B$.
Let $b: X^o:=X\setminus D \to X$ be the open immersion.
Then we have 
\begin{equation}\label{eq cp sga}
\phi b_*b^* (\overline{\mathbb{Q}}_\ell)_X =0.
\end{equation}
\end{lemma}

\begin{proof}
We offer two proofs. 
The first one consist of applying  Beilinson's theorem, to the effect that the vanishing cycle functor $\phi$ commutes with the Verdier duality functor
$\mathbb D$ up to a Tate shift (cf. \cite[Cor. 0.2]{qing-zheng-duality}), i.e. 
there is an isomorphism of functors
$\mathbb D_{X_s} \phi \simeq \phi \mathbb D_X (-1)$.
In fact, since $X$ is smooth over $B,$ the constant sheaf on $X$ is self-dual up to shifts,
so that
the conclusion follows from
the fact that one knows that
 $\phi b_! b^! (\overline{\mathbb{Q}}_\ell)_X =0$
 (cf.  \cite[XIII, LM. 2.1.11, p.105]{SGA7-2}) 
 by an application of Beilinson's theorem together with the standard $\mathbb D b_!b^!=
 b_*b^* \mathbb D.$

For the second one, we argue as follows.
Let $D=\bigcup_{i\in I} D_i$
be the decomposition into irreducible components. For $J\subseteq I$, let $D_J:= \bigcap_{i\in J} D_i.$
By considering the distinguished triangle of functors $(a_!a^!, \rm{id}, b_*b^*)$, since the smoothness of $X/B$ implies that $\phi \overline{\mathbb{Q}}_\ell =0,$
it is enough to show that
$\phi a^! \overline{\mathbb{Q}}_\ell =0.$
 By the Absolute Purity Conjecture 
 proved by Gabber \cite{fujiwara_absolute_purity}, we have that
 if $a_E: E \to X$ is the closed immersion of a pure  codimension $c$ regular subscheme in $X,$ then $a_E^!(\overline{\mathbb{Q}}_\ell)_X= (\overline{\mathbb{Q}}_\ell)_E [-2c](-c),
 $  so that $\phi a_E^! (\overline{\mathbb{Q}}_\ell)_X =0.$ We thus have $\phi a^!_{D_J} (\overline{\mathbb{Q}}_\ell)_{D_J}=0,$ for every $J \subset I.$
 The conclusion follows by a simple devissage argument involving the application of the functors $\phi a^!$ applied to the graded pieces of the stupid filtration on the acyclic complex providing a resolution of the constant sheaf on $D$:
 \[
0 \to (\overline{\mathbb{Q}}_\ell)_D
\to \oplus_{|J|=1}
(\overline{\mathbb{Q}}_\ell)_{D_J}
\to
\oplus_{|J|=2}
(\overline{\mathbb{Q}}_\ell)_{D_J}
\ldots
\oplus_{|J|=|I|}
(\overline{\mathbb{Q}}_\ell)_{D_J}
\to 0.
 \]
\end{proof}

We have the closed and open immersions:
\begin{equation}\label{eq:opcl}
\xymatrix{
 \partial \widetilde{\overline{M{\rm{Hodge}}}} \ar[r]^-a 
 &
\widetilde{\overline{M{\rm{Hodge}}}}
 &
 \overline{M{\rm{Hodge}}} \ar[l]_-b.
} 
\end{equation}

In this paragraph, we work on the two charts before taking the quotient by $\mathbb G_m.$
We have morphisms as in 
(\ref{eq:opcl}). By the smoothness of $M^* \to B$ (cf. Theorem \ref{thm:smooth}), we  know that $\phi \overline{\mathbb{Q}}_\ell =0.$
The boundary is a simple normal crossing divisor on $M^*$ relative to $B$.
By Lemma \ref{cplmt sga},
we obtain the
identity
$\phi b_*b^* \overline{\mathbb{Q}}_\ell=0.$ By the exactness of $\phi$ applied to the distinguished triangle of functors $(a_!a^!, {\rm id}, b_*b^*),$ we see that  $\phi a_!a^! \overline{\mathbb{Q}}_\ell = 0.$ These vanishing of vanishing cycle complexes occur on the two charts before taking the quotient by $\mathbb G_m$.
The purpose of the following lemma is to descend these three identities to the quotient by $\mathbb{G}_m$, i.e. to the two  charts of $\widetilde{\overline{M{\rm{Hodge}}}}$.

\begin{lemma}\label{lm:phib}
Let rank and degree be coprime and let us assume we are in the situation with poles.
We have $\phi  \overline{\mathbb{Q}}_\ell=\phi a_!a^! \overline{\mathbb{Q}}_\ell =\phi b_*b^* \overline{\mathbb{Q}}_\ell=0$ on $\widetilde{\overline{M{\rm{Hodge}}}}.$
\end{lemma}
\begin{proof}
For the first $\phi  \overline{\mathbb{Q}}_\ell = 0$, we apply \cite[Lemma 4.1.5]{decataldo-cambridge}. The second $\phi a_!a^! \overline{\mathbb{Q}}_\ell =0$ would follow from $\phi  \overline{\mathbb{Q}}_\ell=\phi b_*b^* \overline{\mathbb{Q}}_\ell=0$ by applying $\phi$ to the distinguished triangle
$(a_! a^!, \text{Id}, b_* b^*).$ We are therefore left with proving $\phi b_*b^* \overline{\mathbb{Q}}_\ell=0$.

We use the notation of \cite[(72)]{decataldo-cambridge} freely, where 
the quotient morphism $\pi= q \circ p:
M^* \to {M^*}' \to \overline{M{\rm{Hodge}}}$  by $\mathbb{G}_m$ is written as the composition of a quotient by a finite group
(containing the stabilizers of the action)
followed by a quotient by the free residual $\mathbb{G}_m$-action.

We shall show that the desired identity $\phi b_*b^* \overline{\mathbb{Q}}_\ell=0$ holds on the second chart $\overline{M{\rm de \, Rham}}\times \mathbb{A}^1$, 
with quotient map 
$\pi= q \circ p:   (M^*)_{x=1} \times \mathbb{A}^1 \to  \overline{M{\rm de \, Rham}}\times \mathbb{A}^1.$ The same proof applies for the first chart.

We have the following chain of  implications:  (caution, the first identity is on the chart
before taking the quotient, and the last is on the chart itself, i.e. after the application of $\pi = q \circ p$)

$
(\phi b_* b^* \overline{\mathbb{Q}}_\ell=0)
$
(on $(M^*)_{x=1} \times \mathbb{A}^1$)
$
 \Rightarrow
$

$
(p_*(\phi b_* b^* \overline{\mathbb{Q}}_\ell) =0) \Rightarrow
$
$
(\phi  p_* b_* b^* \overline{\mathbb{Q}}_\ell =0) \Rightarrow
$
$
(\phi  b_* p_* b^* \overline{\mathbb{Q}}_\ell =0) \Rightarrow
$

$
(\phi  b_*  b^* p_* \overline{\mathbb{Q}}_\ell =0) \Rightarrow
$
$
(\phi  b_*  b^*  \overline{\mathbb{Q}}_\ell =0) \Rightarrow
$
$
(\phi  b_*  b^* q^* \overline{\mathbb{Q}}_\ell =0) \Rightarrow
$

$
(\phi  b_*  q^* b^* \overline{\mathbb{Q}}_\ell =0) \Rightarrow
$
$
(\phi  q^* b_*  b^* \overline{\mathbb{Q}}_\ell =0) \Rightarrow
$
$
(q^* \phi b_*  b^* \overline{\mathbb{Q}}_\ell =0) \Rightarrow
$

$
(\phi b_*  b^* \overline{\mathbb{Q}}_\ell =0) 
$ (on $\overline{M{\rm de \, Rham}} \times \mathbb{A}^1$),

\noindent
where: the first implication is a mere application of $p_*;$
the second is because $p$ is finite, hence proper, so that $p_*\phi= \phi p_*;$
the third is by the commutativity of \cite[(72)]{decataldo-cambridge}; the fourth is because $b^*=b^!$ for open immersions
and we always have base change $p_*b^!=b^!p_*;$
the fifth is because $\overline{\mathbb{Q}}_\ell$ is a direct summand of $p_* \overline{\mathbb{Q}}_\ell$ (cf.  \cite[Lemma 5.3]{decataldo-zhang-completion} applied to $p$);
the sixth is simply because $q^* \overline{\mathbb{Q}}_\ell = \overline{\mathbb{Q}}_\ell;$ the seventh is
 by the commutativity of \cite[(72)]{decataldo-cambridge};  the eight is because $q$ is smooth of relative dimension one
 so that $q^!$ equals $q^*[2]$ and base change; the ninth is again by the smoothness
 of $q$ since then $\phi q^* = q^* \phi;$ and the tenth is because $q^*$ preserves stalks
 and $q$ is surjective.
\end{proof}

\end{subsection}

\begin{subsection}{Proof of Theorems \ref{thm:spk} and \ref{thm:spv}}\label{subs:pfspk}

\begin{notn}
When we are working in positive equicharacteristic, there is a filtered version
of the statements of Theorems \ref{thm:spk} and \ref{thm:spv}. When we do not wish to repeat verbatim an argument
which has been provided for the unfiltered version in order to 
 prove the filtered version, we resort to locutions such as ``(filtered) isomorphism."
\end{notn}

 \begin{proof}[Proof of Theorems \ref{thm:spk}]
By virtue of
Lemma \ref{lm:phib},
the hypotheses  of the unfiltered version of  \cite[Prop. 3.4.2.(A)]{decataldo-cambridge}
are met when applied to the completion 
$\overline{\tau_k}$ (\ref{cotau}) of the structural morphism $\tau_k$ of the Hodge moduli 
space. 
We deduce that the arrows on the bottom row of (\ref{eq:qapl}) are isomorphisms of cohomology
 rings, that  the specialization morphism is defined, and that  it is an isomorphism of cohomology rings. 
 For the filtered version of the sought-after statement, we use the filtered version of 
\cite[Prop. 3.4.2.(A)]{decataldo-cambridge}.
 
 By applying the same method of proof of \cite[Thm. 3.5]{decataldo-zhang-nahpostive}, we see that we reach the desired conclusions
 for the top row of (\ref{eq:qapl}) (filtered and unfiltered version).
 
 The left-hand-side vertical arrow in (\ref{eq:qapl}) is the identity, hence the sought-after properties are automatically valid.
 
 The right-hand-side vertical arrow, being identified with the morphism associated with an extension of separably closed fields,
 is also a (filtered) isomorphism.
 
 Every arrow  in diagram (\ref{eq:qapl}), except for the middle vertical arrow, is a (filtered) isomorphism,
 forcing the middle vertical arrow to be one as well.
\end{proof}

\begin{proof}[Proof of Theorem \ref{thm:spv}]
The goal is to prove that all the arrows in
 (\ref{iso55}) exist and are (filtered) isomorphisms.

 As a starting point,
we use  the commutative diagram of  non-curved morphisms 
of cohomology rings (\ref{eq:iso055}).
The non-curved arrows in the top and bottom  row of   (\ref{eq:iso055}) are (filtered) isomorphisms of cohomology rings by Theorem \ref{thm:spk} applied to $\tau_{s}$ and to $\tau_{\overline{\eta}}$.  We also have
that the corresponding specialization morphisms on the top and bottom rows
are  defined and are (filtered) isomorphisms.
  
 We use the completion  $\overline{v_{{\rm{Higgs}}}}$ (\ref{cohi}) 
of the Higgs  moduli spaces. Since we have proven smoothness of the morphism $v_{Higgs,B} : M{\rm{Higgs}}_{C_B}^{ss} \to B$, we can check the hypotheses of \cite[Prop. 3.4.2.(A)]{decataldo-cambridge} exactly in the same way as we did for Hodge in Lemma \ref{lm:phib} (note that in this case we don't need to consider the second chart). It follows that
all the non-curved arrows
in the left-hand-side  Higgs column of (\ref{eq:iso055})  are filtered isomorphisms,
and that the corresponding filtered specialization morphism is defined and is a filtered isomorphism. Here we are using universal corepresentability $M{\rm{Higgs}}_{C_B}^{ss}$ (Remark \ref{remark: universal correpresentability}) to identify the special fiber with $M{\rm{Higgs}}_{C_{s}}^{ss}$.

The arrows in the right-hand-side de Rham column of \eqref{eq:iso055} are well-defined isomorphisms (filtered, when $\text{char} B  >0$) by the same argument using the completion $\overline{v_{{\rm de \, Rham}}}$ (\ref{code}) 
of de Rham moduli spaces with poles. 

We use the completion  $\overline{v_{{\rm{Hodge}}}}$ (\ref{eq:lok}) of the Hodge moduli space over $B$.
In view of Lemma \ref{lm:phib}, we can apply   \cite[Prop. 3.4.2.(A)]{decataldo-cambridge} and deduce that
the  middle Hodge column is made of (filtered) isomorphisms and that the (filtered) specialization morphism
for this middle Hodge columns  is defined and is a (filtered) isomorphisms.

We are now left with  showing that the  two horizontal arrows $\rho_{0_B}$ and $\rho_{1_B}$ in the middle row
are (filtered) isomorphisms.
This follows formally from the commutativity of the diagram of non-curved arrows (\ref{eq:iso055}),
 and the fact that all the remaining non-curved arrows 
have been proven to be (filtered) isomorphisms.
\end{proof}
\end{subsection}
\end{section}

\begin{section}{Appendix: Factorization of the $p$-curvature morphism; with Siqing Zhang}\label{appdx}
\begin{centering}
 \large{ \; \;\textbf{Mark Andrea de Cataldo, Andres Fernandez Herrero, Siqing Zhang}}
\end{centering}

 For the definition of the Hodge-Hitchin morphism (\ref{eq:hodgemo1}) in the case of connections without poles, see \cite[Prop. 3.2]{laszlo-pauly-p-curvature}, which works with a curve  over a field. For a stronger result, which covers the case of curves
 --and of higher dimensions as well-- over a Noetherian base, see \cite[Cor. 5.7]{langer-moduli-lie-algebroids}.

 The proof of \cite[Prop. 3.2]{laszlo-pauly-p-curvature}  contains a minor inaccuracy, for it is stated that
 the stack of $t$-connections (without poles) is smooth over the base field, whereas even the open substack of semistables is not smooth. This purported smoothness is used in the proof of loc. cit.

 Theorem \ref{thm:smooth}  proves that the stack of semistable $t$-connections
 with poles is smooth over the Noetherian base $B.$ In this paper, we apply this smoothness result to the case when  $B$ is a field and  when $B$ is a DVR, both of which are reduced. Therefore, 
 the proof given in  \cite[Prop. 3.2]{laszlo-pauly-p-curvature} works, with the trivial modification
stemming from the fact that: while in the case with no poles 
loc. cit. uses, in the context of the elegant 
``Bost's Trick," the elementary identity 
 $\partial_x^{[p]}:= \partial_x \circ \cdots \circ \partial_x$ ($p$ times) $=0$, in the case with poles we can use the identity $(x\partial_x)^{[p]} = x \partial_x.$ In the end,
 while loc. cit. ends with a factor $t$ in the case without poles, we end with a factor $tx$
 in the case with poles, and the logic to reach the desired conclusion, namely the existence of the Hodge-Hitchin morphism for families of curves over a reduced Noetherian scheme, is the same.

 In this appendix, we remove the assumption made above of semistability,
 as well as the assumption on the Noetherian $B$ being reduced. We give a proof of the existence of the Hodge-Hitchin morphism (\ref{eq:hodgemo1})
 from the stack of $t$-connections, with or without poles for a family
 of curves over a Noetherian base $B.$ Moreover, we correct the minor inaccuracy in the proof of
\cite[Prop. 3.2]{laszlo-pauly-p-curvature}.

 The key step is to reduce to an auxiliary family of curves over a suitable
 complete and reduced ring, where then the Lazlo-Pauly logic is valid,
 without poles (factor $t$), and with poles (factor $tx$).
 We now give the details of this key reduction  step.
 
By Noetherian approximation (more precisely: choose a relative polarization and then use the fact that the stack of polarized smooth geometrically connected curves over $\mathbb{F}_p$ is locally of finite presentation over $\mathbb{F}_p$ \cite[\href{https://stacks.math.columbia.edu/tag/0DSS}{Tag 0DSS}]{stacks-project}+ \cite[\href{https://stacks.math.columbia.edu/tag/0CE81}{Tag 0E81}]{stacks-project}+\cite[\href{https://stacks.math.columbia.edu/tag/0DQ0}{Tag 0DQ0}]{stacks-project}, combined with \cite[\href{https://stacks.math.columbia.edu/tag/0CMX}{Tag 0CMX}]{stacks-project} applied to a colimit $B=\colim_i B_i$ as in \cite[\href{https://stacks.math.columbia.edu/tag/01ZA}{Tag 01ZA}]{stacks-project}) the curve $C \to B$ fits into a Cartesian diagram:
\begin{figure}[H]
\centering
\begin{tikzcd}
  C \ar[r] \ar[d] & C' \ar[d] \\  B \ar[r] & B',
\end{tikzcd}
\end{figure}
where $B'$ is of finite type over the prime field $\mathbb{F}_p$ and $C' \to B'$ is a smooth projective morphism with geometrically connected fibers of dimension $1$. This means that the Hodge stacks fit into the following diagram
\begin{figure}[H]
\centering
\begin{tikzcd}
  \mathcal{M}\text{Hodge}(C) \ar[r] \ar[d] & \mathcal{M}\text{Hodge}(C') \ar[d] \\  B \ar[r] & B',
\end{tikzcd}
\end{figure}
Since the formation of the $p$-curvature morphism is compatible with base-change in the base $B$, it suffices to show the desired factorization for $\mathcal{M}\text{Hodge}(C')$, and so we can assume without loss of generality that $B$ is of finite type over $\mathbb{F}_p$.

Since the stack $\mathcal{M}\text{Hodge}(C)$ is locally of finite type over $B$ \cite[Prop 2.2.2]{torsion-freepaper}, it suffices to check the desired factorization for any family over an affine scheme $\on{Spec}(R)$ with $R$ of finite type over $\mathbb{F}_p$. Such a point $\on{Spec}(R) \to \mathcal{M}\text{Hodge}(C)$ corresponds to a function $t \in  R$, a vector bundle $\mathcal{F}$ on $C_R$, and a $t$-connection $\nabla$ on $\mathcal{F}$.

We can write $R = \mathbb{F}_p[t_1, t_2, \ldots t_r]/I$ for some ideal $I$. We denote by $\hat{S}$ the completion of the polynomial ring $\mathbb{F}_p[t_1, t_2, \ldots t_r]$ with respect to the ideal $I$. Since $\mathbb{F}_p[t_1, t_2, \ldots t_r]$ is a reduced G-ring \cite[\href{https://stacks.math.columbia.edu/tag/07PX}{Tag 07PX}]{stacks-project}, the completion $\hat{S}$ is reduced (\cite[\href{https://stacks.math.columbia.edu/tag/0AH2}{Tag 0AH2}]{stacks-project} + \cite[\href{https://stacks.math.columbia.edu/tag/0C21}{Tag 0C21}]{stacks-project}).  Choose a $\on{Spec}(R)$-ample line bundle $\mathcal{L}$ on the family $C$. The deformation theory of smooth curves equipped with ample line bundles is unobstructed (\cite[\href{https://stacks.math.columbia.edu/tag/0AH2}{Tag 0AH2}]{stacks-project}+\cite[\href{https://stacks.math.columbia.edu/tag/0E84}{Tag 0E84}]{stacks-project}). 
Similarly the deformation theory of the vector bundle $\mathcal{F}$ has obstructions in the groups $H^2(C, I^j \otimes \mathcal{E}nd(\mathcal{F})) = 0$ \cite[Thm. 8.5.3]{fga-explained}, and so it 
is unobstructed as well. Therefore we can get a compatible family 
of lifts of the triple $(C, \mathcal{L}, \mathcal{F})$ for 
every nilpotent thickening $\mathbb{F}_p[t_1, t_2, \ldots, t_r]/I^j$ as $j$ ranges over the positive integers. 
By Grothendieck's existence and algebraization theorems (\cite[\href{https://stacks.math.columbia.edu/tag/089A}{Tag 089A}]{stacks-project}+ \cite[\href{https://stacks.math.columbia.edu/tag/0885}{Tag 03O}]{stacks-project}), we can algebraize this formal tuple into families $(\widetilde{C}, \widetilde{\mathcal{F}}, \widetilde{\mathcal{L}})$ over $\on{Spec}(\hat{S})$. Therefore, 
we get a Cartesian diagram of families of smooth curves:
\begin{figure}[H]
\centering
\begin{tikzcd}
  C \ar[r] \ar[d] & \widetilde{C} \ar[d] \\  \on{Spec}(R) \ar[r, symbol = \hookrightarrow] & \on{Spec}(\hat{S}),
\end{tikzcd}
\end{figure}
and a vector bundle $\widetilde{\mathcal{F}}$ on $\widetilde{C}$ 
such that its restriction to $C$ recovers $\mathcal{F}$. Choose 
a lift $\widetilde{t} \in \hat{S}$ of $t \in R$. 
In order to show the factorization of the $p$-curvature morphism 
as in \cite[Prop. 3.2]{laszlo-pauly-p-curvature}, we need to 
show that certain canonically defined sections of powers 
of the line bundle $\omega_{C/S}$ vanish. This can be done Zariski locally on $C.$ 
Choose an affine open covering $\widetilde{U}_i$ of $\widetilde{C}$ 
that trivializes $\widetilde{\mathcal{F}}$. We fix trivializations 
of $\widetilde{\mathcal{F}}|_{\widetilde{U}_i}$. We denote by $U_i$
the restriction to $C$, which yields an affine open covering with 
trivializations of the restriction $\mathcal{F}$. It suffices to 
show that the factorization of the $p$-curvature map on every $U_i$.
The $t$-connection $\nabla$ on the trivial bundle $\mathcal{F}|_{U_i}$ 
can be written as $t d_{U_i} +M$, where 
$d_{U_i}: \mathcal{O}_{U_i} \to \Omega^1_{U_i/\hat{S}}$ 
denotes the exterior derivative on $U_i$ and $M \in H^0(\omega_{U_i/R}^{\oplus n^2})$ is a matrix of differentials. Choose a lift $\widetilde{M} \in H^0(\omega_{\widetilde{U}_i/\hat{S}}^{\oplus n^2})$
of $M$, and define $\widetilde{\nabla}$ to be the $\widetilde{t}$ connection $\widetilde{t} d_{\widetilde{U}_i} + \widetilde{M}$ on the trivial bundle $\widetilde{\mathcal{F}}_{\widetilde{U}_i}$. The $\widetilde{t}$-connection $(\widetilde{\mathcal{F}}|_{\widetilde{U}_i}, \widetilde{\nabla})$ on $\widetilde{U}_i$ restricts to the $t$-connection $(\mathcal{F}|_{U_i}, \nabla)$ under the base-change by $\on{Spec}(R) \hookrightarrow \on{Spec}(\hat{S})$. Since the formation of the $p$-curvature is compatible with such base-change, it suffices to show the desired factorization for the $\widetilde{t}$-connection $(\widetilde{\mathcal{F}}|_{\widetilde{U}_i}, \widetilde{\nabla})$ on $\widetilde{U}_i$. Thus we can work over the reduced ring $\hat{S}$ and on affine open subsets of $\widetilde{C}$ to prove the desired factorization. By passing to each irreducible component of $\on{Spec}(\hat{S})$, we can furthermore assume that $\hat{S}$ is an integral domain. Hence we can use the local computation outlined in the proof of \cite[Prop. 3.2]{laszlo-pauly-p-curvature}, which assumes that the base ring is an integral domain.
The calculation is carried out in the case without poles by using the vector field $\partial_x.$ The case with poles is analogous,
once we replace the vector field  $\partial_x$ with $x\partial_x.$
Note also that the case without poles implies directly the case with poles: the sections are trivial away from the poles, hence are trivial across the poles.
\end{section}

\bibliographystyle{alpha}
\footnotesize{\bibliography{smoothness_hodge.bib}}

\newcommand{\etalchar}[1]{$^{#1}$}
\begin{thebibliography}{HLHJ21}

\bibitem[Alp13]{alper-good-moduli}
Jarod Alper.
\newblock Good moduli spaces for {A}rtin stacks.
\newblock {\em Ann. Inst. Fourier (Grenoble)}, 63(6):2349--2402, 2013.

\bibitem[Alp14]{alper_adequate}
Jarod Alper.
\newblock Adequate moduli spaces and geometrically reductive group schemes.
\newblock {\em Algebr. Geom.}, 1(4):489--531, 2014.

\bibitem[AO21]{alfaya-oliveira-lie-algerbroids}
David Alfaya and André Oliveira.
\newblock Lie algebroid connections, twisted higgs bundles and motives of moduli spaces.
\newblock \url{https://arxiv.org/abs/2102.12246}, 2021.

\bibitem[AOV08]{abramovich2007tame}
Dan Abramovich, Martin Olsson, and Angelo Vistoli.
\newblock Tame stacks in positive characteristic.
\newblock {\em Annales de l'Institut Fourier}, 58(4):1057--1091, 2008.

\bibitem[BNR89]{bnr-spectral-curves}
Arnaud Beauville, M.~S. Narasimhan, and S.~Ramanan.
\newblock Spectral curves and the generalised theta divisor.
\newblock {\em J. Reine Angew. Math.}, 398:169--179, 1989.

\bibitem[BR94]{biswas-ramanan-infinitesimal}
I.~Biswas and S.~Ramanan.
\newblock An infinitesimal study of the moduli of {H}itchin pairs.
\newblock {\em J. London Math. Soc. (2)}, 49(2):219--231, 1994.

\bibitem[BS06]{biswas-subramanian-weil-criterion}
Indranil Biswas and S.~Subramanian.
\newblock Vector bundles on curves admitting a connection.
\newblock {\em Q. J. Math.}, 57(2):143--150, 2006.

\bibitem[BY96]{boden_yokogawa}
Hans~U. Boden and K{\^o}ji Yokogawa.
\newblock Moduli spaces of parabolic {Higgs} bundles and parabolic {{\(K(D)\)}} pairs over smooth curves. {I}.
\newblock {\em Int. J. Math.}, 7(5):573--598, 1996.

\bibitem[CL16]{chaudouard-laumon}
Pierre-Henri Chaudouard and G\'{e}rard Laumon.
\newblock Un th\'{e}or\`eme du support pour la fibration de {H}itchin.
\newblock {\em Ann. Inst. Fourier (Grenoble)}, 66(2):711--727, 2016.

\bibitem[CZ15]{chen-zhu}
Tsao-Hsien Chen and Xinwen Zhu.
\newblock Non-abelian {H}odge theory for algebraic curves in characteristic {$p$}.
\newblock {\em Geom. Funct. Anal.}, 25(6):1706--1733, 2015.

\bibitem[dC17]{de-cataldo-support-sln}
Mark~Andrea de~Cataldo.
\newblock A support theorem for the {H}itchin fibration: the case of {${\rm SL}_n$}.
\newblock {\em Compos. Math.}, 153(6):1316--1347, 2017.

\bibitem[dC22]{decataldo-cambridge}
Mark Andrea~A. de~Cataldo.
\newblock Perverse {L}eray filtration and specialisation with applications to the {H}itchin morphism.
\newblock {\em Math. Proc. Cambridge Philos. Soc.}, 172(2):443--487, 2022.

\bibitem[dCZ21]{decataldo-zhang-nahpostive}
Mark Andrea~A. de~Cataldo and Siqing Zhang.
\newblock A cohomological non abelian {H}odge theorem in positive characteristic.
\newblock \url{https://arxiv.org/abs/2104.12970}, 2021.

\bibitem[dCZ22]{decataldo-zhang-completion}
Mark Andrea~A. de~Cataldo and Siqing Zhang.
\newblock Projective completion of moduli of {$t$}-connections on curves in positive and mixed characteristic.
\newblock {\em Adv. Math.}, 401:Paper No. 108329, 43, 2022.

\bibitem[Del77]{SGA4.5}
P.~Deligne.
\newblock {\em Cohomologie \'{e}tale}, volume 569 of {\em Lecture Notes in Mathematics}.
\newblock Springer-Verlag, Berlin, 1977.
\newblock S\'{e}minaire de g\'{e}om\'{e}trie alg\'{e}brique du Bois-Marie SGA $4\frac{1}{2}$.

\bibitem[DK73]{SGA7-2}
Pierre Deligne and Nicholas Katz.
\newblock {\em Groupes de monodromie en g\'eom\'etrie alg\'ebrique. {II}}.
\newblock Lecture Notes in Mathematics, Vol 340. Springer-Verlag, 1973.
\newblock S{\'e}minaire de G{\'e}om{\'e}trie Alg{\'e}brique du Bois-Marie 1967--1969 (SGA 7, {II}).

\bibitem[FGI{\etalchar{+}}05]{fga-explained}
Barbara Fantechi, Lothar G\"{o}ttsche, Luc Illusie, Steven~L. Kleiman, Nitin Nitsure, and Angelo Vistoli.
\newblock {\em Fundamental algebraic geometry}, volume 123 of {\em Mathematical Surveys and Monographs}.
\newblock American Mathematical Society, Providence, RI, 2005.
\newblock Grothendieck's FGA explained.

\bibitem[Fuj02]{fujiwara_absolute_purity}
Kazuhiro Fujiwara.
\newblock A proof of the absolute purity conjecture (after {G}abber).
\newblock In {\em Algebraic geometry 2000, {A}zumino ({H}otaka)}, volume~36 of {\em Adv. Stud. Pure Math.}, pages 153--183. Math. Soc. Japan, Tokyo, 2002.

\bibitem[Gro16]{groechenig-moduli-flat-connections}
Michael Groechenig.
\newblock Moduli of flat connections in positive characteristic.
\newblock {\em Math. Res. Lett.}, 23(4):989--1047, 2016.

\bibitem[Har10]{hartshorne-deformation}
Robin Hartshorne.
\newblock {\em Deformation theory}, volume 257 of {\em Graduate Texts in Mathematics}.
\newblock Springer, New York, 2010.

\bibitem[HLHJ21]{torsion-freepaper}
Daniel Halpern-Leistner, Andres~Fernandez Herrero, and Trevor Jones.
\newblock Moduli spaces of sheaves via affine {G}rassmannians.
\newblock \url{https://arxiv.org/abs/2107.02172}, 2021.

\bibitem[IIS06]{inaba_iwasaki_saito}
Michi-aki Inaba, Katsunori Iwasaki, and Masa-Hiko Saito.
\newblock Moduli of stable parabolic connections, {R}iemann-{H}ilbert correspondence and geometry of {P}ainlev\'{e} equation of type {VI}. {I}.
\newblock {\em Publ. Res. Inst. Math. Sci.}, 42(4):987--1089, 2006.

\bibitem[Ina13]{inaba_riemann_hilbert}
Michi-Aki Inaba.
\newblock Moduli of parabolic connections on curves and the {R}iemann-{H}ilbert correspondence.
\newblock {\em J. Algebraic Geom.}, 22(3):407--480, 2013.

\bibitem[Lan14]{langer2014semistable}
Adrian Langer.
\newblock Semistable modules over {L}ie algebroids in positive characteristic.
\newblock {\em Doc. Math.}, 19:509--540, 2014.

\bibitem[Lan21]{langer-moduli-lie-algebroids}
Adrian Langer.
\newblock Moduli spaces of semistable modules over {L}ie algebroids.
\newblock \url{https://arxiv.org/abs/2107.03128}, 2021.

\bibitem[LP01]{laszlo-pauly-p-curvature}
Yves Laszlo and Christian Pauly.
\newblock On the {H}itchin morphism in positive characteristic.
\newblock {\em Internat. Math. Res. Notices}, (3):129--143, 2001.

\bibitem[LZ19]{qing-zheng-duality}
Qing Lu and Weizhe Zheng.
\newblock Duality and nearby cycles over general bases.
\newblock {\em Duke Math. J.}, 168(16):3135--3213, 2019.

\bibitem[Mar94]{markman-spectral}
Eyal Markman.
\newblock Spectral curves and integrable systems.
\newblock {\em Compositio Math.}, 93(3):255--290, 1994.

\bibitem[Mar12]{martinego-infinitesimal}
E.~Martinengo.
\newblock Infinitesimal deformations of {H}itchin pairs and {H}itchin map.
\newblock {\em Internat. J. Math.}, 23(7):1250053, 30, 2012.

\bibitem[Moc09]{mochizuki-harmonic-tame-ii}
Takuro Mochizuki.
\newblock Kobayashi-{H}itchin correspondence for tame harmonic bundles. {II}.
\newblock {\em Geom. Topol.}, 13(1):359--455, 2009.

\bibitem[Nit91]{nitsure-higgs}
Nitin Nitsure.
\newblock Moduli space of semistable pairs on a curve.
\newblock {\em Proc. London Math. Soc. (3)}, 62(2):275--300, 1991.

\bibitem[Nit93]{nitsure-log-connections}
Nitin Nitsure.
\newblock Moduli of semistable logarithmic connections.
\newblock {\em J. Amer. Math. Soc.}, 6(3):597--609, 1993.

\bibitem[OV07]{ogus-vologodsky}
A.~Ogus and V.~Vologodsky.
\newblock Nonabelian {H}odge theory in characteristic {$p$}.
\newblock {\em Publ. Math. Inst. Hautes \'{E}tudes Sci.}, (106):1--138, 2007.

\bibitem[Sch98]{schaub-courbes-spectrales}
Daniel Schaub.
\newblock Courbes spectrales et compactifications de jacobiennes.
\newblock {\em Math. Z.}, 227(2):295--312, 1998.

\bibitem[Sim90]{simpson-harmonic-noncompact}
Carlos~T. Simpson.
\newblock Harmonic bundles on noncompact curves.
\newblock {\em J. Amer. Math. Soc.}, 3(3):713--770, 1990.

\bibitem[Sim94a]{Simpson-repnI}
Carlos~T. Simpson.
\newblock Moduli of representations of the fundamental group of a smooth projective variety. {I}.
\newblock {\em Inst. Hautes \'{E}tudes Sci. Publ. Math.}, (79):47--129, 1994.

\bibitem[Sim94b]{simpson-repnII}
Carlos~T. Simpson.
\newblock Moduli of representations of the fundamental group of a smooth projective variety. {II}.
\newblock {\em Inst. Hautes \'{E}tudes Sci. Publ. Math.}, (80):5--79 (1995), 1994.

\bibitem[{Sta}22]{stacks-project}
The {Stacks Project Authors}.
\newblock \textit{Stacks Project}.
\newblock \url{https://stacks.math.columbia.edu}, 2022.

\bibitem[Sun20]{hao-lambda-modules-dm-stacks}
Hao Sun.
\newblock Moduli space of {$\Lambda$}-modules on projective {D}eligne-{M}umford stacks.
\newblock \url{https://arxiv.org/abs/2003.11674}, 2020.

\bibitem[Yok95]{yokogawa-infinitesimal-higgs}
K\^{o}ji Yokogawa.
\newblock Infinitesimal deformation of parabolic {H}iggs sheaves.
\newblock {\em Internat. J. Math.}, 6(1):125--148, 1995.

\end{thebibliography}

  \textsc{Mark Andrea de Cataldo, Department of Mathematics, Stony Brook University},
  \texttt{mark.decataldo@stonybrook.edu}
  
  \par\nopagebreak
  
    \textsc{Andres Fernandez Herrero, Department of Mathematics, Columbia University}, \texttt{af3358@columbia.edu}
    
    \par\nopagebreak
    
     \textsc{Siqing Zhang, Department of Mathematics, Stony Brook University}, \par \nopagebreak \noindent\texttt{siqing.zhang@stonybrook.edu}
  
\end{document}